\documentclass{./styles/svproc}
\usepackage{url}
\usepackage{placeins}
\usepackage{color}
\usepackage{xcolor}
\usepackage[toc,page]{appendix}

\usepackage{latexsym,amsmath, amscd, amssymb}

\newtheorem{assumption}{Assumption}

\usepackage[colorlinks=false]{hyperref}
\usepackage{cleveref}
\crefname{thm}{Theorem}{Theorems}
\crefname{assumption}{Assumption}{Assumptions}

\crefformat{equation}{(#2#1#3)}
\crefname{theorem}{Theorem}{Theorems}
\crefname{proposition}{Proposition}{Propositions}
\crefname{lemma}{Lemma}{Lemmas}
\crefname{collorary}{Collorary}{Colloraries}
\crefname{assumption}{Assumption}{Assumptions}
\crefname{remark}{Remark}{Remark}
\crefname{example}{Example}{Example}
\crefname{definition}{Definition}{Definitions}
\crefname{section}{Section}{Sections}
\crefname{figure}{Figure}{Figures}
\crefname{chapter}{Chapter}{Chapters}
\crefname{table}{Table}{Tables}

\usepackage{mathtools,xparse}
\usepackage{mathrsfs}
\usepackage{enumitem} 
\usepackage{cleveref}
\usepackage{placeins}
\usepackage{bm}

\def\coupMat{{\bm A}_{c}}
\def\heatMat{{\bm A}_{h}}

\def\opnorm{\mathcal{B}}
\def\drift{{\bf a}}
\def\diffusion{{\bf b}}
\newcommand\given[1][]{\:#1\vert\:}

\def\EE{{ \mathbb{E}}}
\def\RR{{\mathbb R}}

\def\trans{{T}}
\def\NN{{\mathbb{N}}}

\def\({\left (}
\def\){\right )}

\def\xDomain{{\Omega_{\x}}}
\def\zDomain{{\Omega_{\z}}}
\def\qDomain{{\Omega_{\q}}}
\def\pDomain{{\Omega_{\p}}}
\def\sDomain{{\Omega_{\s}}}

\def\bigO{{O}}
\def\bT{\Theta}

\def\EE{{ \mathbb{E}}}
\def\RR{{\mathbb R}}

\def\qDet{{ \mathscr{D}_{\q}}}
\def\zDet{{ \mathscr{D}_{\z}}}
\def\qStoch{{ \mathscr{G}_{\q}}}
\def\zStoch{{ \mathscr{G}_{\z} }}

\def\K{{ \bf K }}

\DeclarePairedDelimiter{\abs}{\lvert}{\rvert}
\DeclarePairedDelimiter{\norm}{\lVert}{\rVert}
\NewDocumentCommand{\normLinfty}{ s O{} m }{%
  \IfBooleanTF{#1}{\norm*{#3}}{\norm[#2]{#3}}_{L_\infty}%
}
\NewDocumentCommand{\normLtwo}{ s O{} m }{%
  \IfBooleanTF{#1}{\norm*{#3}}{\norm[#2]{#3}}_{L_2}%
}
\NewDocumentCommand{\normEuc}{ s O{} m }{%
  \IfBooleanTF{#1}{\norm*{#3}}{\norm[#2]{#3}}_{2}%
}
\NewDocumentCommand{\normFrob}{ s O{} m }{%
  \IfBooleanTF{#1}{\norm*{#3}}{\norm[#2]{#3}}_{{\rm F}}%
}
\def\qh{\tilde{\bm q}}
\def\ph{\tilde{\bm p}}
\def\bu{{\bm u}}
\def\bv{{\bm v}}
\def\bf{\bm}
\def\W{{\bf W}}

\def\I{{\bf I}}

\def\G{{\bf G}}

\def\dd{{\rm d}}

\def\q{{ \bf q}}

\def\M{{ \bf M}}
\newcommand{\etabf}{{\boldsymbol \eta}}
\def\p{{ \bf p}}
\def\s{{ \bf s}}
\def\x{{ \bf x}}
\def\z{{ \bf z}}

\def\Q{{ \bf Q}}
\def\C{{ \bf C}}
\def\0{{ \bf 0}}

\def\Lc{\mathcal{L}}

\def\Kc{\mathcal{K}}

\DeclareMathOperator*{\argmax}{arg\,max}
\def\lspan{\mathrm{lin}}

\def\Lcgq{\mathcal{\Lc}_{\rm GLE}} 

\def\etaq{\widetilde{\etabf}}
\def\Kbf{{\bf K}}
\def\Kq{\widetilde{\Kbf}}

\def\Gammabf{{  \bm{\Gamma} }} 
\def\Sigmabf{{  \bm{\Sigma} }} 
\def\Gammaq{{\bm{\widetilde{\Gamma}}}} 

\def\Sigmaq{{  \bm{\widetilde{\Sigma}} }}

\DeclarePairedDelimiter{\inner}{\langle}{\rangle}

\def\I{\bm I}

\usepackage{chngcntr}
\usepackage{apptools}
\AtAppendix{\counterwithin{lemma}{section}}
\AtAppendix{\counterwithin{proposition}{section}}
\AtAppendix{\counterwithin{theorem}{section}}
\AtAppendix{\counterwithin{assumption}{section}}

\begin{document}
\mainmatter              
\title{Ergodic properties of quasi-Markovian generalized Langevin equations with configuration dependent noise and non-conservative force}
\titlerunning{Generalized Langevin equation}  
%
\author{ Benedict Leimkuhler\inst{1} \and Matthias Sachs\inst{2}
}
\authorrunning{Leimkuhler and Sachs} 
%
\tocauthor{Benedict Leimkuhler and Matthias Sachs}
\institute{The School of Mathematics and the Maxwell Institute of Mathematical Sciences, James Clerk Maxwell Building, University of Edinburgh, Edinburgh EH9 3FD\\
\email{b.leimkuhler@ed.ac.uk},
\and
Department of Mathematics, Duke University, Box 90320, Durham NC 27708, USA; and the
Statistical and Applied Mathematical Sciences Institute (SAMSI), Durham NC 27709, USA \\
\email{msachs@math.duke.edu},
}

\maketitle              

\begin{abstract}
We discuss the ergodic properties of quasi-Markovian stochastic differential equations, providing general conditions that ensure existence and uniqueness of a smooth invariant distribution and exponential convergence of the evolution operator in suitably weighted $L^{\infty}$ spaces, which implies the validity of central limit theorem for the respective solution processes. The main new result is an ergodicity condition for the generalized Langevin equation with configuration-dependent noise and (non-)conservative force.

\end{abstract}
\section{Introduction}\label{sec:introduction}
Generalized Langevin equations (GLE) arise from model reduction and have many applications such as sampling of molecular systems \cite{Ceriotti2009,Ceriotti2010,Morrone2011,gle4md,Wu2015}, atom-surface scattering \cite{DoDi1976}, anomalous diffusion in fluids \cite{HoMc2017}, modeling of polymer melts \cite{li2017computing}, chromosome segmentation in e coli \cite{LaKuWiSp2015}, and the modelling of coarse grained particle dynamics \cite{Givon2004,Li2015a}.
The GLE is a non-Markovian formulation, meaning that the evolution of the current state depends not only on the state itself but on the state history. The system is typically formulated with memory terms describing friction with the environment and stochastic forcing.  The presence of memory complicates both the analysis of the equation and its numerical solution.   In this article, we recall the derivation of the GLE as the result of Mori-Zwanzig reduction of large system to model the dynamics of a subset of the variables.   We consider the ergodicity of the equation (existence of a unique invariant distribution and exponential convergence of the associated semigroup in a suitably weighted $L^{\infty}$ space), providing conditions for its validity in case the coefficients of friction and noise depend directly on the reduced position variables.   

\subsection{The generalized Langevin equation}
Consider the situation of an open system exchanging energy with a heat bath.  If there is a strong time scale separation between the dynamics of the heat bath and the explicitly modelled degrees of freedom, the  exchange of energy between these two systems is well modelled by a Markovian process, i.e., dynamic observables such as transport coefficients and first passage times  can be  well reproduced by a simple Markovian approximation of the heat bath. 

By contrast, if we consider a system consisting of a distinguished particle surrounded by collection of particles of approximately the same mass, then a reduced model where the  interaction between the distinguished particle and the solvent particles is replaced by a simple Langevin equation would lead to a poor approximation of the dynamics of the distinguished particle.

In such modelling situations it is necessary to explicitly incorporate memory effects, i.e., non-Markovian random forces and history dependent dissipation. The framework in which such models are typically formulated is that of the generalized Langevin equation. In this article we consider two different types of generalized Langevin equations, both of which are of the form of a stochastic integro differential equation and as such can be viewed as non-Markovian stochastic differential equation (SDE) models.




Let $\qDomain \in \{ \mathbb{R}^{n}, \mathbb{T}^{n}\}$, where $ \mathbb{T}^{n}= \left( \mathbb{R} / \mathbb{Z}\right)^{n} $ denotes the $n$-dimensional standard torus.\footnote{The assumption that configurations are restricted to the torus eliminates several technical complications and is motivated by the frequent applications of GLEs in molecular modelling, where such a formulation is commonly used.} We first consider a generalized Langevin equation of the form
\begin{equation}\label{eq:GLE:nonmark}
\begin{aligned}
\dot{\q} &=  \M^{-1}\p, \\
\dot {\p} &=  {\bm F}(\q) -  \int_{0}^{t}\K(t-s)\M^{-1}\p(s) \dd s  + \etabf(t).
\end{aligned}
\end{equation}
where the dynamic variables $\q \in \qDomain, \p \in \RR^{n}$ denote the configuration variables  and conjugate momenta of a  Hamiltonian system with energy function 
\begin{equation}\label{eq:hamiltonian:gen}
H(\q,\p) = U(\q) + \frac{1}{2} \p^{\trans}\M^{-1}\p,
\end{equation}
where the mass tensor $\M \in \mathbb{R}^{n\times n}$ is required to be symmetric positive definite and $U \in \mathcal{C}^{\infty}(\qDomain, \RR)$ is a smooth potential function so that ${\bm F} = -\nabla U$ constitutes a conservative force.
$\K : [0,\infty) \rightarrow \mathbb{R}^{n\times n }$ is a matrix-valued  function of $t$, which is referred to as the memory kernel, and $\etabf $ is a stationary Gaussian process taking values in $\mathbb{R}^{n}$ and which (in equilibrium) is assumed to be statistically independent of $\q$ and $\p$. We refer to $\etabf$ as the noise process or random force. We further assume that a fluctuation-dissipation relation between the random force $\etabf$ and the memory kernel holds so that
\begin{enumerate}[label=(\roman*)]
\item \label{item:eta:1}
the random force $\etabf$ is unbiased, i.e.,
\[
\EE[\etabf(t)] = \0,
\]
for all $t \in [0,\infty)$.
\item \label{item:eta:2}
the auto-covariance function of the random force and the memory kernel $\K$ coincide up to a constant prefactor, i.e., 
\[
\EE[\etabf(s+t) \etabf^{\top}(s)] = \beta^{-1} \K(t), ~~\beta >0,
\]
where the constant $\beta>0$ corresponds to the inverse temperature of the system under consideration.
\end{enumerate}
\subsubsection{Position dependent memory kernels and non-conservative forces.}
To broaden the range of applications for our model, 
we also consider instances of the generalized Langevin equation where:
\begin{enumerate}[label=(\roman*)]
\item 
the force ${\bm F}$ is allowed to be non-conservative, i.e., it does not necessarily correspond to the gradient of a potential function,
\item
the random force is a non-stationary process.
\end{enumerate}
More specifically, we consider the case where the strength of the random force depends on the value of the configurational variable $\q$, i.e.,
\begin{equation}\label{eq:gle:q:nonmark:1} 
\begin{aligned}
\dot{\q}(t) &= \M^{-1}\p(t), \\
\dot{\p}(t) &= {\bm F}(\q(t)) -  \Kq(\q,t) * \p + \etaq(t).
\end{aligned}
\end{equation}
where ${\bm F} \in \mathcal{C}^{\infty}(\qDomain,\RR^{n})$ is a smooth vector field, and the random force $\etaq$ is assumed to be of the form 
\[
\etaq(t) = g^{\trans}(\q(t)) \etabf(t),
\]
with $\etabf$ again satisfying \ref{item:eta:1} and \ref{item:eta:2} and the convolution term, $\Kq(\q,t) * \p$,
 is of the form 
\[
\Kq(\q,t) * \p =  g^{\trans}(\q(t))\int_{0}^{t}  \Kbf(t-s)g(\q(s)) \p(s) \dd s,
\]
with $g \in \mathcal{C}^{\infty}(\RR^{n},\RR^{n\times n})$ and $\Kbf$ as specified above. We motivate the above described type of non-stationary random force and position dependent dissipation term at the end of the following section. \\

The generic form of the above described GLEs can be derived using a Mori-Zwanzig reduction of  the combined Hamiltonian dynamics of an explicit heat bath representation and the system of interest \cite{Zwanzig1961b,Zwanzig1973,Mori1965}. In what follows, we  briefly outline the Mori-Zwanzig formalism in a simplified setup  following the presentation in \cite{Givon2004}. We will then consider the particular case of the Kac-Zwanzig model and demonstrate how the above instances of the GLE can be derived from this model. 


\subsection{Formal derivation of the generalized Langevin equation via Mori-Zwanzig projection}
Consider an ordinary differential equation of the form 
\begin{equation}\label{eq:gen:ODE}
\begin{aligned}
\dot{\bu} =& f(\bu,\bv),\\
\dot{\bv} =& g(\bu,\bv),
\end{aligned}
\end{equation}
subject to the initial condition 
\begin{equation}
(\bu(0),\bv(0)) = (\bu_{0},\bv_{0}),
\end{equation}
where $f,g$ are smooth functions, i.e., $f \in \mathcal{C}^{\infty}(\RR^{n_{\bu} \times n_{\bv}}, \RR^{n_{\bu}}),g \in \mathcal{C}^{\infty}(\RR^{n_{\bu} \times n_{\bv}}, \RR^{n_{\bv}})$, with $n_{\bv},n_{\bu}$ being  positive integers. Also, assume that there is a probability measure $\mu(\dd \bu, \dd \bv) = \rho(\bu,\bv)\dd \bu \dd \bv$ with smooth density $\rho \in \mathcal{C}^{\infty}(\RR^{n_{\bu} \times n_{\bv}}, [0,\infty))$, which can be associated with a stationary state\footnote{in the sense that $\Lc \rho = 0$, with $\Lc$ being the Liouville operator associated with \cref{eq:gen:ODE}.} of the system \cref{eq:gen:ODE}.
Consider now the projection operator $\mathscr{P}$, which maps observables  $w(\,\cdot\,,\,\cdot\,)$ onto the conditional expectation $\mathscr{P} \bu \mapsto \EE_{\mu} [ w(\bu,\bv) \given \bv]$, i.e., 
\[
\left(\mathscr{P}w\right)(\bu) = 
\frac{
\int_{\RR^{n_{\bv}}}\rho(\bu,\bv) w(\bu,\bv)\dd \bu \dd \bv
}{
\int_{\RR^{n_{\bv}}}\rho(\bu,\bv) \dd \bu \dd \bv
}.
\]
The Mori-Zwanzig projection formalism allows 
to recast the system \cref{eq:gen:ODE} as an integro-differential equation (IDE) of the generic form 
\begin{equation}\label{eq:integro:generic}
\dot{\bu}(t) = \bar{f}(\bu(t)) + \int_{0}^{t}K(\bu(t-s),s) \dd s + \eta(\bu(0),\bv(0),t),
\end{equation}
where $\bar{f} = \mathscr{P}f$, 
$K : \RR^{n_{\bu}} \times [0,\infty) \rightarrow  \RR^{n_{\bu}}$ is a memory kernel,
and $\eta$ is a function of the initial values of $\bu,\bv$ and the time variable $t$. It is important to note that while $\eta$ depends on the initial condition of both $\bu$ and $\bv$ in \cref{eq:gen:ODE}, the remaining terms in the IDE \cref{eq:integro:generic} only depend explicitly on the dynamic variable $\bu$. Similarly as in the stochastic IDEs \cref{eq:GLE:nonmark} and \cref{eq:gle:q:nonmark:1}  the convolution term in \cref{eq:integro:generic} can,  under appropriate conditions on $f,g$, be considered as a dissipation term. Likewise, under the assumptions that $\bu,\bv$ are initialized randomly according to $\mu$, the term $\eta(\bu(0),\bv(0),t)$ in \cref{eq:integro:generic} can be interpreted as a random force.\\

A particularly well studied case is the situation where the functions $f$ and $g$ are such that $(f^{\trans},g^{\trans})^{\trans}$ is a Hamiltonian vector field and \cref{eq:gen:ODE} corresponds to the equation of motion of a Hamiltonian system. In this case a natural choice for $\mu$ is the Gibbs-Boltzmann distribution associated  with the Hamiltonian. This choice of $\mu$ allows us to interpret  the degrees of freedom represented by the dynamical variable $\bv$ as a heat bath or energy reservoir. For example, let $\bu = (\q,\p)\in \RR^{2n}$, $\bv = (\qh,\ph)\in \RR^{2m}$ with $ 2n=n_{\bu},2m = n_{\bv}$. We may consider the case where $f$ and $g$ are derived from the Hamiltonian 
\begin{equation}\label{eq:ham:heatbath}
H(\q,\p,\qh,\ph) = V(\q) +\frac{1}{2} \p^{\trans}\M^{-1}\p +V_{c}(\q,\qh) + V_{h}(\qh) + \frac{1}{2} \ph^{\trans}\widetilde{\M}^{-1}\ph,
\end{equation}
where $V,V_{c},V_{h}$ are smooth potential functions such that $V+V_{c}+V_{h}$ is confining and $\M \in\RR^{n\times n}, \widetilde{\M}\in\RR^{m\times m}$ are symmetric positive definite matrices. In view of \cref{eq:integro:generic} the variables $(\q,\p)$ correspond to the explicitly resolved part of the system; the variables $(\qh,\ph)$ correspond to the part of the system which is ``projected out'' and is replaced by the dissipation term and the fluctuation term, thus it functions as the heat bath in the reduced model.
The coupling between heat bath and explicitly resolved degrees of freedom is encoded in the form of the coupling potential $V_{c}$, and the statistical properties of the heat bath are determined both by the form of the mass matrix $\widetilde{\M}$ and the form of the potential $V_{h}$.\\

Let $P$ denote the projection $(\bu,\bv) \mapsto \bu$. The first step in the derivation of the  IDE \cref{eq:integro:generic}  is to rewrite the first line in \cref{eq:gen:ODE} as  
\begin{equation}\label{eq:gen:ODE:deriv}
\dot{\bu}(t) = \left(\mathscr{P}f\right)(P(\bu(t),\bv(t))) + \left [ f(\bu(t),\bv(t)) -  \left(\mathscr{P}f\right) \left ( P(\bu(t),\bv(t)) \right ) \right ].
\end{equation}
Obviously, the first term in \cref{eq:gen:ODE:deriv} corresponds exactly to $\bar{f}(\bu(t))$ in \cref{eq:integro:generic}. 
Let 
\[
\Lc =  f(\bu,\bv) \cdot \nabla_{\bu} + g(\bu,\bv)\cdot \nabla_{\bv} 
\]
denote the Liouville operator associated with \cref{eq:gen:ODE}. Noting that 
\[
\Lc \left ( P (\bu,\bv) \right ) =  f(\bu,\bv),
\]
 the term in the square brackets in \cref{eq:gen:ODE:deriv} can be rewritten  in semi-group notation as
\begin{equation}\label{eq:gen:ODE:semigroup}
\begin{aligned}
 f(\bu(t),\bv(t)) -  \left(\mathscr{P}f\right) (\bu(t),\bv(t))
 &= e^{t\Lc}(\I - \mathscr{P})f(\bu(0),\bv(0))\\
& = e^{t\Lc}(\I-\mathscr{P}) \Lc P (\bu(0),\bv(0)),
 \end{aligned}
\end{equation}
where $e^{t\Lc}$ denotes the flow-map operator associated with the solution of \cref{eq:gen:ODE}, which is defined so that $e^{t\Lc}w(\bu(0),\bv(0)) = w(\bu(t),\bv(t))$. The integro-differential form \cref{eq:integro:generic} then follows by applying the operator identity
\[
e^{t\Lc} =  \int_{0}^{t} e^{(t-s)\Lc} \mathscr{P} \Lc e^{s(\I-\mathscr{P})\Lc} \dd s +e^{t(\I-\mathscr{P}) \Lc},
\]
which is known as Dyson's formula \cite{morriss2013statistical}, to the last line in \cref{eq:gen:ODE:semigroup} yielding
\begin{equation}\label{eq:stoch:integro:generic:formal}
\begin{aligned}
e^{t\Lc}(\I-\mathscr{P}) \Lc P (\bu(0),\bv(0)) &=  \int_{0}^{t} e^{(t-s)\Lc} \mathscr{P} \Lc e^{s(\I-\mathscr{P})\Lc}(\I-\mathscr{P}) \Lc P (\bu(0),\bv(0)) \dd s\\
&+e^{t(\I-\mathscr{P}) \Lc} (\I-\mathscr{P}) \Lc P (\bu(0),\bv(0)),
 \end{aligned}
\end{equation}
where the second term on the right hand side can be identified with $\eta$ in \cref{eq:integro:generic}, and the first term in \cref{eq:stoch:integro:generic:formal} corresponds to the integral term in \cref{eq:integro:generic}.
The form of the last term in 
\cref{eq:stoch:integro:generic:formal} suggests that $\eta$ can be formally written as the solution of a differential equation
\begin{equation}\label{eq:orthogonal}
\begin{aligned}
\frac{\partial}{\partial t} \eta(\bu(0),\bv(0),t) &= (\I - \mathscr{P})\Lc \eta(\bu(0),\bv(0),t),\\
\eta(\bu(0),\bv(0),0) &= f(\bu(0),\bv(0)) - (\mathscr{P}f) (\bu(0)),
\end{aligned}
\end{equation}
which is commonly referred to as the {\em orthogonal dynamics equation} \cite{darve2009computing,Givon2004}.\\

A couple of remarks are in order. First, we reiterate that the above calculations are purely formal, i.e., the above expressions for the memory kernel $K$ and the fluctuation term $\eta$ in general do not possess a closed form solution and are therefore often considered as intractable. Moreover, the well-posedness of the orthogonal dynamics equation \cref{eq:orthogonal} is not obvious and care needs to be taken regarding the existence of solutions and the interpretation of the differential operator $\Lc$ therein. We refer here to \cite{givon2005existence} for a rigorous treatment of this equation.
We also mention that the above choice of the projection operator $\mathscr{P}$ as a linear operator which maps functions of $(\bu,\bv)$  into the space of functions of $\bu$  constitutes a special case of the Mori-Zwanzig formalism. More general forms of the projection operator $\mathscr{P}$ can be considered within the Mori-Zwanzig formalism. For example, the Mori-Zwanzig formalism can be used to  derive an IDE for the dynamics of reaction coordinates (collective variables). The corresponding projection operator $\mathscr{P}$ is typically nonlinear in these cases, which can drastically complicate the derivation and the form of the IDE.  For a more general presentation of the Mori-Zwanzig projection formalism we refer to the above mentioned papers \cite{darve2009computing,Givon2004} and the references therein as well as the original papers by Mori \cite{Mori1965} and Zwanzig \cite{Zwanzig1961b,Zwanzig1973}. In particular the latter paper by Zwanzig considers nonlinear forms of the projection operator $\mathscr{P}$.\\

Secondly, we point out that in order to derive  the stochastic IDEs  \cref{eq:GLE:nonmark} and \cref{eq:gle:q:nonmark:1} an additional step is required. While \cref{eq:GLE:nonmark} and \cref{eq:gle:q:nonmark:1}  are of the form of a stochastic IDE, i.e., they are IDEs driven by a (non-Markovian) stochastic process, the equation \cref{eq:integro:generic}  constitutes an IDE with random initial data, i.e., the system follows a deterministic trajectory after initialization. In the physics literature it is common, in the situation where $f,g$ define a Hamiltonian vector field, to establish equivalence of these systems by virtue of an averaging argument which is considered valid when the system is in equilibrium and
$n_{\bv}$  is sufficiently large (see e.g. \cite{Kantorovich2008}).  

Drawing a mathematically rigorous connection between \cref{eq:integro:generic} and a  stochastic IDE which resembles the form of \cref{eq:GLE:nonmark} or \cref{eq:gle:q:nonmark:1} requires substantial work. As we discuss in the section below, weak convergence as $n_{\bv} \rightarrow \infty$ of the trajectory of $\bu$ on finite time intervals to the solution of a stochastic integro-differential has been shown in \cite{Kupferman2002,Kupferman2004} for instances of the Ford-Kac model. 
\subsubsection{The Ford-Kac model.}
We consider the Mori-Zwanzig projection formalism in the situation where the ODE \cref{eq:gen:ODE} corresponds to the equation of motion derived from the Hamiltonian \cref{eq:ham:heatbath}. We already mentioned above that the memory kernel $K$ and the fluctuation term in the IDE \cref{eq:integro:generic}  in general do not possess a closed form solution. A notable exception, however, is the situation of a linearly coupled harmonic heat bath, e.g.,
\begin{equation}\label{eq:FK:coupling}
V_{c}(\q,\qh) = \q^{\trans}\coupMat\qh,
\end{equation}
 with $\coupMat\in \RR^{n\times m}$, and  
\begin{equation}\label{eq:FK:heat}
V_{h}(\qh) = \frac{1}{2} \qh^{\trans}\heatMat\qh,
\end{equation}
with $\heatMat\in \RR^{m\times m}$ being a symmetric positive (semi-)definite matrix. Under this choice of the potential functions $V_{c}$ and $V_{h}$, the equations of motion associated with \cref{eq:ham:heatbath} are of the form
\begin{equation}\label{eq:kac:motion}
\begin{aligned}
\dot{\q} &= \M^{-1} \p,\\
\dot{\p} &= -\nabla_{\q} V(\q)  + \coupMat \qh, \\
\dot{\qh} &= \widetilde{\M}^{-1}\ph,\\
\dot{\ph} &= -\heatMat\qh  + \coupMat^{\trans}\q.
\end{aligned}
\end{equation}
The system \cref{eq:kac:motion} was first studied in \cite{ford1965statistical} and is commonly referred to as {\em Ford-Kac model}.
Integrating the 3rd and 4th line of \cref{eq:kac:motion} we obtain
\begin{equation}\label{eq:FK:integral}
\begin{pmatrix} \qh(t) \\ \ph(t) \end{pmatrix} =
e^{t{\bm R}}\begin{pmatrix} \qh(0) \\ \ph(0) \end{pmatrix}
 + \int_{0}^{t} e^{{(t-s) {\bm R} }} \begin{pmatrix} \0 \\ \coupMat^{\trans}\q(s) \end{pmatrix} \dd s,
\end{equation}
where by ${\bm R} \in \RR^{2m\times 2m}$ we denote the matrix
\[
{\bm R} = \begin{pmatrix} \0 & \widetilde{\M}^{-1} \\
 -\heatMat & 0 \end{pmatrix}.
\]
Partial integration of the integral term in \cref{eq:FK:integral} yields
\[
\begin{aligned}
\begin{pmatrix} \qh(t) \\ \ph(t) \end{pmatrix}  &=
e^{t{\bm R}}\begin{pmatrix} \qh(0) \\ \ph(0) \end{pmatrix} + {\bm R}^{-1}  \begin{pmatrix} \0 \\ \coupMat^{\trans}\q(t) \end{pmatrix} 
- {\bm R}^{-1}e^{{t {\bm R} }} \begin{pmatrix} \0 \\ \coupMat^{\trans}\q(0) \end{pmatrix} + \int_{0}^{t} e^{{(t-s) {\bm R} }} \begin{pmatrix} \0 \\ \coupMat^{\trans}\p(s)  \end{pmatrix} \dd s.
\end{aligned}
\]
Substituting $\tilde{\q}$ in the 2nd line by this expression we obtain an IDE of the form \cref{eq:integro:generic} with the deterministic vector field $\bar{f}$ being of the form
\[
\bar{f}(\q,\p) = \begin{pmatrix} \M^{-1}\p \\ - \nabla_{\q}V(\q) -  \coupMat \heatMat \coupMat^{\trans}\q  \end{pmatrix},
\]
the memory kernel $K$ being of the form 
\begin{equation}\label{KF:kernel}
K(\p(t-s),s) =  
- \begin{pmatrix} \0 & \0 \\ \0 & \coupMat^{-1}  \end{pmatrix}
 e^{{(t-s) {\bm R} }} \begin{pmatrix} \0 \\ \coupMat^{\trans}\p(s)  \end{pmatrix},
\end{equation}
and the fluctuation term being of the form
\begin{equation}\label{KF:noise}
\eta(\qh(0),\ph(0),\q(0),t) = e^{t{\bm R}}\begin{pmatrix} \qh(0) \\ \ph(0) \end{pmatrix} - {\bm R}^{-1}e^{{t {\bm R} }} \begin{pmatrix} \0 \\ \coupMat^{\trans}\q(0) \end{pmatrix}.
\end{equation}

\subsubsection{The thermodynamic limit of the Ford-Kac model.}
A detailed analysis of the thermodynamic limit $m\rightarrow \infty$ of an instance of the Ford-Kac model can be found in \cite{Kupferman2002}; see also \cite{Kupferman2004,Givon2004}. The Hamiltonian of the system considered in \cite{Kupferman2002} comprises a single distinguished particle of unit mass, which is subject to an external force associated with the confining potential function $U \in\mathcal{C}^{\infty}(\RR,\RR)$. The heat bath is modeled by $m$ particles. Each of the heat bath particles is attached by a linear spring to the distinguished particle. The heat bath particles are not subject to any additional force apart from the coupling force. The corresponding Hamiltonian can be written\footnote{One easily verifies that this Hamiltonian corresponds to a parametrization of \cref{eq:ham:heatbath} as  $\M=1,~\widetilde{\M}= {\rm diag}(\tilde{m}_{1},\dots,\tilde{m}_{m}),~V(\q) = U(\q) + \frac{1}{2}\sum_{i=1}^{m} k_{i}\q^{2},~V_{c}(\q,\qh) = \sum_{i=1}^{m} k_{i}\q\qh_{i},~ V_{h}(\qh) = \frac{1}{2}\sum_{i=1}^{m}k_{i}\qh_{i}^{2}$.}
\begin{equation}\label{eq:Ham:FK}
H(\q,\p,\qh,\ph)
= \frac{1}{2}\p^{2} + U(\q) + \frac{1}{2} \sum_{j=1}^{m}\frac{\ph_{j}^{2}}{\tilde{m}_{j}}+\frac{1}{2}\sum_{j=1}^{m}k_{j}(\qh_{j}-\q),  
\end{equation}
where $k_{j}>0$ corresponds to the stiffness constant of the spring attached to the $j$-th heat bath particle and $\tilde{m}_{j}>0$ is the mass of the $j$-th heat bath particle.
For this system one finds that the terms \cref{KF:kernel} and \cref{KF:noise} take a particular simple form, so that the corresponding IDE can be written as
\begin{equation}\label{eq:integro:diff:harmonic}
\begin{aligned}
\dot{\q} &= \p,\\
\dot{\p} &= -\partial_{\q}U(\q)   - \int_{0}^{t} K^{(m)}(t-s)\p(s) \dd s + \eta^{(m)}(\qh_{i}, \ph_{i},t),
\end{aligned}
\end{equation}
where the memory kernel is of the form 
\[
K^{(m)}(t) = \sum_{i=1}^{m}k_{i} \cos(\omega_{i}t),
\]
and the fluctuation term is of the form
\[
\eta^{(m)}(\qh_{i}, \ph_{i}, t) = \sum_{i=1}^{m}\sqrt{\frac{k_{i}}{\beta}} \Big( \qh_{i}(0) \cos(\omega_{i}t) +  \ph_{i}(0)\sin(\omega_{i}t) \Big ),
\]
with $\omega_{j} =\sqrt{k_{j}/\tilde{m}_{j}}$. 
If the initial conditions of the heat bath particles are assumed to be distributed according to the Gibbs measure associated with \cref{eq:Ham:FK} and the statistical distribution of the values of $k_{j}$ and $\tilde{m}_{j}$ are controlled in a certain way as $m\rightarrow \infty$, it can been shown that for any finite $T>0$ the trajectories  of the solution of \cref{eq:integro:diff:harmonic} converge weakly within the interval $[0,T]$ to  solutions of a stochastic IDE of the form \cref{eq:GLE:nonmark}; for a precise statement see \cite[Theorem 4.1]{Kupferman2002}.
\subsubsection{The Kac-Zwanzig model}
The Kac-Zwanzig model (see \cite{Zwanzig1973}) is a generalization of the Ford-Kac model, the heat bath is still harmonic, i.e., $V_{h}$ has the general form \cref{eq:FK:heat}, but the coupling potential is such that the coupling force is linear in $\qh$ but non-linear in $\q$, i.e.,
\[
V_{c}(\q,\qh) = {\bm G}(\q) \qh,
\]
where ${\bm G}\in \mathcal{C}^{2}(\RR^{n},\RR^{n\times m})$. For such a system a closed form solution of the terms in the Mori-Zwanzig projection \cref{eq:integro:generic} can still be derived (see \cite{Zwanzig1973} or \cite{hanggi1997generalized} for a detailed derivation). However, unlike in the situation of the Ford-Kac model 
the closed form solution of the memory kernel $K$ and the fluctuation term $\eta$ are functions of $\q$. This observation motivates the study of GLEs of the form \cref{eq:gle:q:nonmark:1}. Instances of \cref{eq:gle:q:nonmark:1} which are derived from such a Kac-Zwanzig heat bath model can be found for example in \cite{Kantorovich2008,stella2014generalized,Ness2015,Ness2016}.\\

We note that an elegant alternative derivation of the GLE can be obtained beginning from a model of an infinite-dimensional heat-bath. Such models have been extensively studied in \cite{jakivsc1997ergodic,jakvsic1997spectral,jakvsic1998ergodic}, and in a (non-equilibrium) context by Rey-Bellet and coworkers in \cite{Eckmann1998,Eckmann1999a,rey2003statistical,Rey-Bellet2006b}.

\subsection{Main results and organization of the paper}
In this article we focus on instances of the GLEs \cref{eq:GLE:nonmark} and \cref{eq:gle:q:nonmark:1} (or, more precisely, \cref{eq:gle:q:nonmark}), which can be represented in an extended phase space as an It\^o diffusion process. We refer to such GLEs, which possess a Markovian representation in an extended phase space as quasi-Markovian generalized Langevin equations (QGLEs). We specify the extended variable formalism, i.e., the particular form of the It\^o diffusion processes which we consider for a Markovian representation of GLEs, in the following \cref{sec:extend:var:1}. In that section we also review results from the literature on the Markovian representation and approximation of generalized Langevin equations. 
The main results of this article are contained in \cref{thm:ergo:bounded:1,thm:ergo:unbounded:2,thm:ergo:unbounded:3,thm:ergo:GLEq} which we present in \cref{sec:ergodic:GLE}. In these theorems we provide criteria which ensure geometric ergodicity for the Markovian representation of GLEs of the form \cref{eq:GLE:nonmark} and \cref{eq:gle:q:nonmark:1}. Since the extended variable formalism which we consider in this article is in various ways more general than the extended variable formalisms considered for ergodicity proofs in previous works in the literature our results cover a wide class of GLEs, which have previously not been shown to be (geometrically) ergodic and which are of high interest in applications (for a detailed discussion see the notes at the end of \cref{sec:notes}). In particular, showing (geometric) ergodicity for QGLEs with non-conservative forces and/or stated dependent memory kernels is a novel contribution of this paper. As a consequence of the geometric ergodicity we can derive in a generic way the validity of a central limit theorem (see \cref{cor:CLT}) for the solution processes of the respective GLEs. For the proofs of the \cref{thm:ergo:bounded:1,thm:ergo:unbounded:2,thm:ergo:unbounded:3,thm:ergo:GLEq} suitable Lyapunov functions must be constructed and the validity of a minorization condition ensured; see \cref{sec:lem:stat} and \cref{sec:lem:nonstat} for details, and \cref{sec:stochastic:analysis} for a general overview of the employed framework. For the proof on the existence of suitable Lyapunov functions we use a similar ansatz as in previous works (compare in particular with \cite{Mattingly2002,Ottobre2011}), but we require additional linear algebra arguments due to increased generality of our extended variable formalism. The proof of the validity of the minorization condition in the case of position dependent coefficients requires a non-standard alternation of the common techniques. We show the existence of a minorizing measure by virtue of a Girsanov transformation. 




\section{Markovian representation of generalized Langevin equations with configuration dependent noise}\label{sec:extend:var:1}
 
In this section we derive a Markovian representation of the GLEs introduced in \cref{sec:introduction}. We start with an It\^o diffusion process of the form
\begin{equation}\label{eq:gle:q}
\begin{aligned}
\dot{\q} &= \M^{-1}\p,\\
\dot{\p} & = {\bm F}(\q) -\Gammaq_{1,1}(\q)\M^{-1}\p - \Gammaq_{1,2}(\q) \s +\beta^{-1/2}\Sigmaq_{1}(\q) \dot{\W}, \\
\dot{\s} &= - \Gammaq_{2,1}(\q) \M^{-1}\p - \Gammaq_{2,2}(\q) \s + \beta^{-1/2}\Sigmaq_{2}(\q) \dot{\W},\\
\text{with}~ &\big (\q(0),\p(0),\s(0) \big ) \sim \mu_{0},
\end{aligned}
\end{equation}
where $\M, {\bm F}, \beta$ are as previously defined. In particular ${\bm F}$ may correspond to the negative gradient of a smooth and confining potential function $U\in \mathcal{C}^{\infty}(\qDomain,\RR)$, i.e., ${\bm F} = -\nabla U$. Furthermore,
\begin{enumerate}[label=(\roman*)]
\item the auxiliary variable $\s(t)$ takes values in $\RR^{m}$ with $m\geq n$,\\
\item $\dot{\W} = [\dot{W}_{i} ]_{1\leq i \leq n+m}$ is a vector of $(n+m)$ independent Gaussian white-noise components, i.e.,
$\dot{W}_{i} \sim \mathcal{N}(0,1)$ and $\EE [ \dot{W}_{i}(t) \dot{W}_{j}(s) ] = \delta_{ij} \delta(t-s)$.\\
\item $\Gammaq_{i,j}, \Sigmaq_{i}, \,i=1,2$ are matrix valued functions so that for $m\geq n$,
\[
\Gammaq = 
\begin{pmatrix}
\Gammaq_{1,1} &\Gammaq_{1,2}\\
\Gammaq_{2,1} & \Gammaq_{2,2}
\end{pmatrix} \in \mathcal{C}^{\infty}\(\qDomain,\RR^{(n+m)\times (n+m)}\).\]
and
\[
\Sigmaq = \begin{pmatrix} \Sigmaq_{1,1} &\Sigmaq_{1,2} \\ \Sigmaq_{2,1} & \Sigmaq_{2,2} \end{pmatrix} = \begin{pmatrix} \Sigmaq_{1} \\ \Sigmaq_{2}\end{pmatrix} \in \mathcal{C}^{\infty}\(\qDomain,\RR^{(n+m)\times (n+m)}\),
\]
 i.e., 
\[
\Gammaq_{1,1}\in \mathcal{C}^{\infty}(\qDomain, \RR^{n\times n}),~\Gammaq_{2,1}^{\trans},\Gammaq_{1,2}\in (\qDomain, \RR^{n\times m}),~ \Gammaq_{2,2}\in\mathcal{C}^{\infty}(\qDomain, \RR^{m\times m}),
\]
and
 \[
 \Sigmaq_{1} \in  \mathcal{C}^{\infty}(\qDomain, \RR^{n\times (n + m)}),~\Sigmaq_{2}  \in \mathcal{C}^{\infty}(\qDomain, \RR^{m\times (n+m)}).
\]
\item \label{it:gle:q:4}
The probability measure $\mu_{0}$ is such that $(\q(0),\p(0),\s(0))$ has finite first and second moments. In particular,
\[
\int_{\qDomain \times \RR^{n+m}} \norm {\q}^{2}_{2} +  \norm {\p}^{2}_{2}  + \norm {\s}^{2}_{2}  ~\mu_{0} (\dd \q, \dd \p, \dd \s) < \infty.
\]
\end{enumerate}
\subsubsection{Notation.}
In the sequel, we write $\x^{\trans} := (\q^{\trans},\p^{\trans},\s^{\trans})$, as well as $\z^{\trans} := (\p^{\trans},\s^{\trans})$ as shorthand notation for the phase space and auxiliary variables, and we use $\xDomain := \qDomain \times \pDomain \times \sDomain$, and  $\zDomain := \pDomain \times \sDomain$, where $\pDomain=\RR^{n}, \sDomain =\RR^{m}$, as shorthand notation for the corresponding domains. With some abuse of notation we also denote points in $\xDomain,\zDomain, \qDomain,\pDomain,\sDomain$ by $\x,\z,\q,\p,\s$, respectively. 

\subsubsection{Associated generator.} \label{subsec:generator}
We denote the generator of \cref{eq:gle:q} by 
\begin{equation}\label{def:gle:generator}
\Lcgq = \Lc_{H} + \Lc_{O},
\end{equation}
where $\Lc_{H}$ and $ \Lc_{O}$, which when considered as operators on $\mathcal{C}^{\infty}(\xDomain,\RR)$, have the form
\[
\Lc_{H} := {\bm F}(\q)\cdot \nabla_{\p} + \M^{-1}\p \cdot \nabla_{\q},
\]
and
\[
\Lc_{O} := - \Gammaq(\q)\begin{pmatrix} \M^{-1}\p \\ \s\end{pmatrix} \cdot \nabla_{\z} + \frac{\beta^{-1}}{2} \Sigmaq(\q) \Sigmaq^{\trans}(\q)  : \nabla^{2}_{\z},
\]
where 
\[
 \Sigmaq(\q) \Sigmaq^{\trans}(\q)  : \nabla^{2}_{\z} = \sum_{i=1}^{M} \sum_{j=1}^{M}\left [  \Sigmaq(\q) \Sigmaq^{\trans}(\q)  \right ]_{i,j} \partial_{\z_{i}}\partial_{\z_{j}}, ~M=n+m.
\]
\subsubsection{Derivation of the associated stochastic IDE.}
In what follows we relate the system \eqref{eq:gle:q}  to a non-Markovian stochastic IDE. 
Consider the following convolution functional 
\begin{gather}\begin{aligned}
\Kq_{\Gammaq}(\q,t) * \p &= 
\Gammaq_{1,1}(\q(t))\M^{-1} \p(t)\\ &~~- \Gammaq_{1,2}(\q(t)) \int_{0}^{t} \exp \left ({-\int_{s}^{t}\Gammaq_{2,2}(\q(r)) \dd r }\right ) \Gammaq_{2,1}(\q(s))\M^{-1}\p(s) \dd s,
\end{aligned}\raisetag{2.5\baselineskip}
\end{gather}
and a random force of the form
\[
\etaq(t) =  \etaq_{w}(t) + \etaq_{c}(t),
\]
where 
\begin{equation}\label{eq:def:wn}
\etaq_{w}(t) := \beta^{-1/2}\Sigmaq_{1}(\q(t))\dot{\W}(t),
\end{equation}
and
\begin{equation}\label{eq:def:cn}
\etaq_{c}(t) := - \Gammaq_{1,2}(\q(t)) \etabf_{c}(t),
\end{equation}
with $\etabf_{c}$ being the solution of the linear SDE
\begin{equation}\label{eq:gle:q:noise}
\dot{\etabf}_{c}(t) = - \Gammaq_{2,2}(\q(t)) \etabf_{c}(t) +  \beta^{-1/2}\Sigmaq_{2}(\q(t)) \dot{\W}(t),~~ \etabf_{c}(0) = \s(0).\\
\end{equation}
As shown in the following proposition, under this assumption, the SDE \eqref{eq:gle:q} can  be rewritten as a stochastic IDE of the form
\begin{equation}\label{eq:gle:q:nonmark} 
\begin{aligned}
\dot{\q}(t) &= \M^{-1}\p(t), \\
\dot{\p}(t) &= {\bm F}(\q(t)) -  \Kq_{\Gammaq}(\q,t) * \p + \etaq(t).
\end{aligned}
\end{equation}
\begin{proposition}\label{prop:GLE:mark}
If a (weak) solution of $(\q(t),\p(t),\s(t))$ of \cref{eq:gle:q} exists for all times $t\geq0$, the SDE \eqref{eq:gle:q} can be rewritten in the form \eqref{eq:gle:q:nonmark}.
\begin{proof}
The solution for $\s$ in \cref{eq:gle:q} can be written as
\begin{equation}\label{eq:s:solution}
\begin{aligned}
\s(t) = {\bm \Phi}(t,0,\q) \s(0) &- \int_{0}^{t} {\bm \Phi}(t,s,\q) \Gammaq_{2,1}(\q(s))\M^{-1}\p(s) d s\\
&+ \int_{0}^{t}{\bm \Phi}(t,s,\q)\Sigmaq_{2}(\q(s)) \dd \W(s), 
\end{aligned}
\end{equation}
with 
\begin{equation}
{\bm \Phi}(t,s,\q) = \exp\left ({-\int_{s}^{t}\Gammabf_{2,2}(\q(r)) d r } \right ).
\end{equation}
Substituting $\s(t)$ in the second equation of \cref{eq:gle:q} by the right hand side of \cref{eq:s:solution} we obtain
\begin{equation*}
\begin{aligned}
\dot{\p}(t) 
&= {\bm F}(\q(t)) -\Gammaq_{1,1}(\q(t)) \M^{-1}\p(t) \\
&~~+ \Gammaq_{1,2}(\q(t))\int_{0}^{t}{\bm \Phi}(t,s,\q)\Gammaq_{2,1}(\q(s))\M^{-1}\p(s)\dd s- \Gammaq_{1,2}(\q(t)){\bm \Phi}(t,0,\q)\s(0)\\
&~~ -\Gammaq_{1,2}(\q(t))\int_{0}^{t}{\bm \Phi}(t,s,\q)\Sigmaq_{2}(\q(s))\dd{\W}(s)+ \Sigmaq_{1}(\q(t))\dd{\W}(t).\\
\end{aligned}
\end{equation*}
As the solution of \eqref{eq:gle:q:noise}, $\etabf_{c}(t)$  can be written as
\[
\etabf_{c}(t) ={\bm \Phi}(t,0,\q) \s(0)-\Gammaq_{1,2}(\q(t))\int_{0}^{t}{\bm \Phi}(t,s,\q) \Sigmaq_{2}(\q(s))\dd \W(s), 
\]
and we find:
\[
\begin{aligned}
\dot{\p}(t) 
&=
{\bm F}(\q(t)) -  \Kq_{\Gammaq}(\q,t) * \p + \etaq_{w}(t)  -\Gammaq_{1,2}(\q(t)) \etabf_{c}(t)\\
&={\bm F}(\q(t)) -  \Kq_{\Gammaq}(\q,t) * \p + \etaq(t).
\end{aligned}
\]
\qed\end{proof}
\end{proposition}

\begin{example}[Quasi-Markovian GLE with constant coefficients]\label{ex:gle:1}
If we consider the case where $\Gammaq$ and $\Sigmaq$ are constant, i.e., $\Gammaq \equiv \Gammabf$ and $\Sigmaq \equiv \Sigmabf$ with $\Gammabf, \Sigmabf  \in \mathbb{R}^{(n+m)\times(n+m)}$, one finds that the convolution term simplifies to
\[
 \Kq_{\Gammaq}(\q,t) * \p = -\Gammabf_{1,1} \M^{-1}\p(t) + \int_{0}^{t}\Gammabf_{1,2}e^{-\Gammabf_{2,2}(t-s)}\Gammabf_{2,1}\M^{-1}\p(s)\dd s,
\]
and the noise terms become
\begin{equation}
\begin{aligned}
\etaq_{w}(t)=\Sigmabf_{1}\dot{\W}(t), ~\etaq_{c}(t)=- \Gammabf_{1,2}e^{-\Gammabf_{2,2}t}\s(0)-\Gammabf_{1,2}\int_{0}^{t}e^{-\Gammabf_{2,2}(t-s)}\Sigmabf_{2} \dd \W(s),
\end{aligned}
\end{equation}
so that the stochastic IDE \cref{eq:gle:q:nonmark} resembles the form of the GLE \cref{eq:GLE:nonmark} with 
\begin{equation}\label{eq:def:K}
\K(t) =  \delta(t)\Gammabf_{1,1} + \Gammabf_{1,2}e^{-\Gammabf_{2,2}(t-s)}\Gammabf_{2,1}.
\end{equation}
\end{example}
\begin{example}[Quasi-Markovian GLE with position dependent noise strength]\label{ex:gle:2}
If we consider the case where $\Gammaq_{2,2}$ and $\Sigmaq_{2,2}$ are constant, i.e., $\Gammaq_{2,2} \equiv \Gammabf_{2,2}$ and $\Sigmaq_{2,2} \equiv \Sigmabf_{2,2}$ with $\Gammaq,\Sigmaq  \in \mathbb{R}^{m\times m}$, the convolution term simplifies to 
\begin{equation}
 \Kq_{\Gammaq}(\q,t) * \p =  \Gammaq_{1,2}(\q(t))\int_{0}^{t}e^{-\Gammabf_{2,2}(t-s)}\Gammaq_{2,1}(\q(s))\M^{-1}\p(s)\dd s,
\end{equation}
and the random force terms $\etaq_{w}$ and $\etaq_{c}$ become
\begin{equation}
\begin{aligned}
\etaq_{w}(t)= \Sigmaq_{1}(\q(t))\dot{\W}(t), 
\end{aligned}
\end{equation}
and 
\begin{equation}
\etaq_{c}(t)=- \Gammaq_{1,2}(\q(t))e^{-\Gammabf_{2,2}t}\s(0) -\Gammaq_{1,2}(\q(t))\int_{0}^{t}e^{-\Gammabf_{2,2}(t-s)}\Sigmaq_{2}(\q(s))\dd \W(s).
\end{equation}
so that for $m=n$ and $\Gammaq_{1,2} = -\Gammaq_{2,1}^{\trans}$, $\Sigmaq_{1,2} = \Sigmaq_{2,1}^{\trans} \equiv \0$, the stochastic IDE \cref{eq:gle:q:nonmark} resembles the form of the GLE \cref{eq:gle:q:nonmark:1} with $\K(t) = e^{-\Gammabf_{2,2}t}$.  
\end{example}
\begin{remark}[Existence of solutions of \cref{eq:gle:q}]
A sufficient condition for \cref{eq:gle:q} to possess a unique strong solution $\x(t)$ for all times $t\geq 0$, is that the right hand side of the SDE \cref{eq:gle:q} is Lipschitz in $\q,\p,\s$. Provided that the initial state $\mu_{0}$ is as specified in  \ref{it:gle:q:4}, it directly  follows by standard existence and uniqueness results for SDEs (see e.g. \cite[Theorem 5.2.1.]{Oksendal2003}) that 
for any $T>0$ there exists a unique strong solution $\x(t), t \in [0,T]$ of \eqref{eq:gle:q}, which is continuous in $t$  and 
\[
\EE \left [\int_{0}^{T}\norm{\x(t)}_{2}^{2}\dd t \right ] < \infty.
\]
Since  ${\bm F},\Gammaq,\Sigmaq$ are assumed to be smooth the Lipschitz condition is obviously satisfied for $\qDomain = \mathbb{T}^{n}$. Similarly, for an unbounded configurational domain, i.e., $\qDomain=\RR^{n}$, the Lipschitz condition for the right hand side of \cref{eq:gle:q} follows directly if the spectra of $\Gammaq(\q)$ and $\Sigmaq(\q)$ are uniformly bounded in $\q$ and ${\bm F}$ satisfies certain asymptotic growths conditions (e.g., \cref{as:potential:unbounded:2}). We also note that the existence of suitable Lyapunov functions as derived in, e.g.,  \cref{prop:lya:R} is sufficient (see e.g. \cite{Rey-Bellet2006a}) to ensure the existence of a weak solution $(\x(t))_{t\geq 0}$ under less strict asymptotic growth conditions on the force ${\bm F}$. 
\end{remark}
\subsection{Fluctuation-dissipation relation for quasi-Markovian generalized Langevin equations}
The following assumption can be understood as a fluctuation dissipation relation for the SDE \cref{eq:gle:q}:
\begin{assumption}\label{as:gle:q}
There exists a symmetric positive definite matrix $\Q \in\RR^{m\times m}$ such that for all $\q \in \qDomain$,
\begin{equation}\label{eq:Lyapunov:Gamma:q}
\Gammaq(\q) 
\begin{pmatrix}
\I_{n} & \0 \\
\0 & \Q
\end{pmatrix}
 + 
 \begin{pmatrix}
\I_{n} & \0 \\
\0 & \Q
\end{pmatrix}
\Gammaq^{\trans}(\q)  = \Sigmaq(\q)  \Sigmaq^{\trans}(\q).
\end{equation}
\end{assumption}
As shown in  \cref{prop:stat:noise}, below, for a quasi-Markovian GLE with constant coefficients (see \cref{ex:gle:1}),  \cref{as:gle:q} implies that the random force is stationary with covariance function $\K$ as specified in \cref{eq:def:K}.
\begin{proposition}\label{prop:stat:noise}
Let as in \cref{ex:gle:1} $\Gammaq$ and $\Sigmaq$ be constant, i.e., $\Gammaq \equiv \Gammabf$ and $\Sigmaq \equiv \Sigmabf$ with $\Gammabf, \Sigmabf  \in \mathbb{R}^{(n+m)\times(n+m)}$. 
If \cref{as:gle:q} is satisfied and $\mu_{0}$ such that $\s(0) \sim \mathcal{N}(\0, \Q)$, where $\Q \in \mathbb{R}^{m\times m}$ as specified in \cref{as:gle:q}, then $\etaq$ is a stationary Gaussian process with vanishing expectation and covariance function $\K$ as defined in \cref{eq:def:K}.
\begin{proof}
Let
\[
\G(r) =\Gammabf_{1,2} \int_{0}^{r}e^{-\Gammabf_{2,2}(r-s)}\Sigmabf_{2}  d\W(s).
\]
Without loss of generality we assume that $t\geq t'$, and we find that the covariance of $\etaq$ is indeed of the form \cref{eq:def:K}: 
\begin{equation*}
\begin{aligned}
\EE \left [\etaq(t)\etaq^{\trans}(t') \right] 
=& \;\EE \left[ \Sigmabf_{1}\dot{\W}(t) \dot{\W}(t')^{\trans} \Sigmabf_{1}^{\trans}] - \EE[ \G(t)  \dot{\W}^{\trans}(t')\Sigmabf_{1}^{\trans} \right] \\
&+ \EE \left [\Gammabf_{1,2}e^{-\Gammabf_{2,2} t} \s(0)\s(0)^{\trans} e ^{-\Gammabf_{2,2}^{\trans} t'} \Gammabf_{1,2}^{\trans}  \right ]\\
&+\EE \left [  \left (\Gammabf_{1,2} \int_{0}^{t'}e^{-\Gammabf_{2,2}(t-s)}\Sigmabf_{2}  d\W(s)\right ) \G^{\trans}(t')  \right ] \\
=& \; \delta(t-t')  (\Gammabf_{1,1} +  \Gammabf_{1,1}^{\trans})- \Gammabf_{1,2}e^{-\Gammabf_{2,2}(t-t')}(\Gammabf_{2,1} + \Q \Gammabf_{1,2}^{\trans})\\
&+ \Gammabf_{1,2}e^{-\Gammabf_{2,2} t} \Q e ^{-\Gammabf_{2,2}^{\trans} t'} \Gammabf_{1,2}^{\trans}  \\
&+ \Gammabf_{1,2}\int_{0}^{t'} e^{-\Gammabf_{2,2}(t-s)}(\Gammabf_{2,2}\Q + \Q \Gammabf_{2,2}^{\trans}) e^{-\Gammabf_{2,2}^{\trans}(t'-s)}\Gammabf_{1,2}^{\trans} d s \\
= & \; \delta(t-t')  (\Gammabf_{1,1} +  \Gammabf_{1,1}^{\trans})- \Gammabf_{1,2}e^{-\Gammabf_{2,2}(t-t')}\Gammabf_{2,1},
\end{aligned}
\end{equation*}
where expectations are taken over both $\mu_{0}$ and the path measure of the Wiener process $\W$. The last equality follows by partial integration of the integral term.
\qed\end{proof}
\end{proposition}
In the absence of a white-noise component in the random force, i.e., $\Gammaq_{1,1},\Sigmaq_{1,1}\equiv \0$, together with the requirement of $\Gammaq(\q)$ to be stable for all $\q\in \qDomain$, \cref{as:gle:q} imposes a constraint on the form $\Gammaq_{1,2}$ and $\Gammaq_{2,1}$ as shown in the following  \cref{lem:purecolor}.
\begin{proposition}\label{lem:purecolor}
Let $\Gammaq,\Sigmaq, \Q$ be such that the conditions of \cref{prop:GLE:cons:q} are satisfied. $\Gammaq_{1,1} \equiv \0$ implies
 \begin{equation}\label{eq:noise:purecolor}
\forall \,\q \in \qDomain:~\Gammaq_{1,2}(\q)\Q = - \Gammaq_{2,1}^{\trans}(\q).
\end{equation}
\begin{proof}
Writing \cref{eq:Lyapunov:Gamma:q} in terms of the sub-blocks of $\Gammaq$ we find
\begin{equation}\label{eq:Lyapunov:Gamma:q:sub}
\begin{pmatrix}
\0  & \Gammaq_{1,2}(\q) \Q + \Gammaq_{2,1}^{\trans}(\q) \\
\Q \Gammaq_{1,2}^{\trans}(\q)  + \Gammaq_{2,1}(\q)  & \Gammaq_{2,2}(\q) \Q + \Q \Gammaq_{2,2}^{\trans}(\q) 
\end{pmatrix}
=\Sigmaq(\q) \Sigmaq^{\trans}(\q) .
\end{equation}
By \cref{lem:schur:posdef} (iii) it follows that the left hand side of \cref{eq:Lyapunov:Gamma:q:sub} is a positive semi-definite matrix for all $\q\in\qDomain$ if and only if \cref{eq:noise:purecolor} holds.
\qed\end{proof}
\end{proposition}
\subsubsection{Equilibrium generalized Langevin equation}
In the particular case of a conservative force, i.e., ${\bm F} = -\nabla U$, one can easily derive a closed form solution for an invariant measure of the SDE \cref{eq:gle:q} if \cref{as:gle:q} holds:
\begin{proposition}\label{prop:GLE:cons:q}
Let ${\bm F} = -\nabla U$, and let \cref{as:gle:q} hold. The SDE \eqref{eq:gle:q}  conserves the probability measure $\mu_{\Q,\beta}(d\x)$ with density
\begin{equation}\label{eq:invariant:measure}
\rho_{\Q,\beta}(\x)\propto e^{-\beta [ U(\q) + \frac{1}{2}\p^{\trans} \M^{-1}\p + \frac{1}{2}\s^{\trans} \Q^{-1} \s]}.\end{equation}
\begin{proof}
The statement follows by inspection of the stationary Fokker-Planck equation associated with the SDE \eqref{eq:gle:q}.
\qed\end{proof}
\end{proposition}
\subsection{Non-equilibrium quasi-Markovian generalized Langevin equations without fluctuation-dissipation relation}
In general one might also consider instances of \cref{eq:gle:q}, where a fluctuation dissipation relation in the form of \cref{as:gle:q} does not hold. Such situations might appear in the modelling of temperature gradients or swarming/flocking phenomena; see, e.g., \cite{sachs2017langevin} for Markovian variants of such models. For example, one may consider an instance of \cref{eq:gle:q}, where $\Gammaq$ and $\Sigmaq$ are of the form
\begin{equation}
\Gammaq =
\begin{pmatrix}
\hat{\Gammabf}_{1,1}^{(1)} &  \hat{\Gammabf}^{(1)}_{1,2} & \hat{\Gammabf}^{(2)}_{1,2}\\
\hat{\Gammabf}^{(1)}_{2,1} & \hat{\Gammabf}^{(1)}_{2,2} & \0\\
\0 & \0 &\hat{\Gammabf}^{(2)}_{2,2}
\end{pmatrix}, ~~
\Sigmaq = 
\begin{pmatrix}
\hat{\Sigmabf}_{1,1}^{(2)} & \0 & \0\\
\0 & \0 & \0\\
\0 & \0 &\hat{\Sigmabf}^{(2)}_{2,2}
\end{pmatrix},
\end{equation}
where 
\[
\hat{\Gammabf}_{1,1}^{(1)},\hat{\Sigmabf}_{1,1}^{(1)}\in \mathcal{C}^{\infty}(\qDomain,\RR^{n\times n }),  
~~~
\hat{\Gammabf}^{(1)}_{1,2}, \left (\hat{\Gammabf}^{(1)}_{2,1} \right)^{\trans}, \hat{\Gammabf}^{(2)}_{1,2} \in \mathcal{C}^{\infty}(\qDomain,\RR^{n\times \hat{m} }),
\]
and
\[
\hat{\Gammabf}^{(1)}_{2,2}, \hat{\Gammabf}^{(2)}_{2,2}, \hat{\Sigmabf}^{(2)}_{2,2} \in \mathcal{C}^{\infty}(\qDomain,\RR^{\hat{m}\times \hat{m}}),
\]
with $\hat{m} \in \NN$ such that $m = 2\hat{m}$ and $\hat{m} \geq n$.
One can easily verify that in the view of the corresponding non-Markovian form \cref{eq:gle:q:nonmark}, the coefficients $\hat{\Gammabf}_{i,j}^{(1)}, 1\leq i,j \leq 2 $ determine the statistical properties of the dissipation, i.e., the form of the convolution functional $\Kq_{\Gammaq}(\q,t) * \p$, and the coefficients $\hat{\Gammabf}_{1,2}^{(2)},\hat{\Gammabf}_{2,2}^{(2)}$ and $\hat{\Sigmabf}_{1,1}^{(2)},\hat{\Sigmabf}_{2,2}^{(2)}$ determine the statistical properties of the random force   $\etaq$. As a simple example we mention the case where the coefficients $\hat{\Gammabf}_{i,j}^{(k)}$ and $\hat{\Sigmabf}_{i,j}^{(k)}$ are constant, i.e., 
\[
\hat{\Gammabf}_{1,1}^{(1)},\hat{\Sigmabf}_{1,1}^{(1)}\in \RR^{n\times n },  \hat{\Gammabf}^{(1)}_{1,2}, \left (\hat{\Gammabf}^{(1)}_{2,1} \right)^{\trans}, \hat{\Gammabf}^{(2)}_{1,2} \in \RR^{n\times m },\hat{\Gammabf}^{(1)}_{2,2}, \hat{\Gammabf}^{(2)}_{2,2}, \hat{\Sigmabf}^{(2)}_{2,2} \in \RR^{m\times m},
\]
with $\hat{\Sigmabf}^{(2)}_{2,2} =   \hat{\Gammabf}^{(2)}_{2,2}+ \left (\hat{\Gammabf}^{(2)}_{2,2}\right )^{\trans}$.
Under suitable conditions  on these matrices (compare with the respective conditions stated in the preceding sections), it can then be easily shown that the SDE \cref{eq:gle:q} can be rewritten as
\begin{equation}
\begin{aligned}
\dot{\q} &= \M^{-1} \p,\\
\dot{\p} &= {\bm F}(\q) - \int_{0}^{t}\K_{1} (t-s) \p(s) \dd s + \etaq,
\end{aligned}
\end{equation}
where 
\begin{equation}
\K_{1}(t) =  \delta(t)\hat{\Gammabf}^{(1)}_{1,1} -   \hat{\Gammabf}^{(1)}_{1,2} e^{-t \hat{\Gammabf}^{(1)}_{2,2}}\hat{\Gammabf}^{(1)}_{2,1},
\end{equation}
and $\etaq$ is a stationary Gaussian process with covariance function $\K_{2}$ of the form 
\begin{equation}
\K_{2}(t) = 2\delta(t) \hat{\Sigmabf}^{(2)}_{1,1} + \hat{\Gammabf}^{(2)}_{1,2} e^{-t \hat{\Gammabf}^{(2)}_{2,2}}\left(\hat{\Gammabf}^{(2)}_{1,2} \right )^{\trans}.
\end{equation}

\subsection{Markovian representations of the GLE in the literature}\label{sec:review:markovian}
In the special case of $\Gammaq,\Sigmaq$ being constant (see \cref{ex:gle:1}), the Markovian representation \cref{eq:gle:q}  is of similar generality to that  presented in \cite{Ceriotti2010,Lei2016a} and the steps in the derivation are essentially the same (see also \cite[Chapter 8]{Pavliotis2014}). 
Likewise, a derivation of a Markovian representation of the form  \cref{eq:gle:q} can for example be found in a slightly less general setup in \cite{lim2017homogenization}. 
We point out that besides the above mentioned generic frameworks, there are many Markovian representations of the GLE mentioned in the literature which are derived in the context of a particular physical model or application. For example, the Markovian representations of the GLE derived in \cite{Doll1976,adelman1976generalized,Kupferman2004,Rey-Bellet2006b} can be considered as special instances of the SDE \cref{eq:gle:q} with constant coefficients $\Gammaq,\Sigmaq$. 
Similarly, some of the non-equilibrium models studied in \cite{Eckmann1998,Eckmann1999,Eckmann1999a,rey2003statistical,Rey-Bellet2006b} can be represented in the form of  \cref{eq:gle:q} with constant coefficients $\Gammaq,\Sigmaq$. 
Markovian representations of the GLE with position dependent memory kernels, which can be viewed as instances of the SDE \cref{eq:gle:q} can be found in \cite{Kantorovich2008,Ness2015,Ness2016,li2017computing}.


\subsubsection{Sufficient condition for the existence of a Markovian representation}
Let $\etabf$ be a real-valued stationary Gaussian process with vanishing mean and covariance function $\K \in \mathcal{C}(\RR, \RR)$, i.e.,
\[
\forall s,t \in \RR, ~\EE[\etabf(t)] = 0,~ \K(t) = \EE[\etabf(s+t)\etabf(s)].
\] 
We denote by  $\widehat{\mu}_{\K}$ the spectral measure of $\K$, i.e.,
\[
\K(t) = \int_{\RR}e^{ikt} \dd \widehat{\mu}_{\K}(k). 
\]
Note that the existence of the spectral measure is a direct consequence of the following proposition, which is an adapted (and simplified) version of what is commonly referred to as Bochner's theorem.
\begin{proposition}\label{prop:bochner}
A complex-valued function $C$ with domain $\RR$ is the covariance function of a continuous weakly stationary\footnote{A stochastic process $(X(t))_{t\in\RR}$ with associated covariance function $C$ is said to be weakly stationary if $\EE[X(t)] = \EE[X(t+s)]=0$ and $C(0,s)= C(t,t+s)$ for all $t,s \in\RR$. Since Gaussian processes are fully characterized by the mean and covariance function, a Gaussian processes is weakly stationary if and only if it is stationary. } random process on $\RR^{n}$ with finite first and second moments, if and only if it can be represented as
 \[
C(t) =\int_{\RR}e^{i t  k } \dd \mu (k),
\]
where $\mu$ is a positive finite measure.
\end{proposition}
The above \cref{prop:bochner} is a simplified version of \cite{stein2012interpolation}. For a proof of the theorem we refer to any standard text book in Fourier analysis, such as \cite[Chapter 1]{rudin2017fourier}.\\

Assume that $\widehat{\mu}_{\K}$ possesses a density with respect to the Lebesgue measure, i.e.,
\[
\widehat{\mu}_{\K}(\dd k)= \widehat{\rho}_{\K}(k)\dd k.
\]
It has been observed in \cite{rey2003statistical} (see also \cite{Eckmann1998,Rey-Bellet2006b} for similar results), that $\left ( \widehat{\rho}_{\K}(k) \right )^{-1}$ being polynomial implies that $\etabf$ can be rewritten as a Markov process in an extended phase space. This can be seen as a consequence of the following criteria for Markovianity: 
\begin{proposition}\label{prop:markov:criteria:general}
If $p(k) =  \sum_{m=1}c_{m}(-ik)^{m}$ is a polynomial with real coefficients and roots in upper half plane then the Gaussian process with spectral density $\abs{p(k)}^{-2}$ is the solution of the stochastic differential equation
\[
p \left (- i \frac{\dd}{\dd t} \right ) \etabf(t)\dd t = \dd \W(t)
\]
\end{proposition}
The above proposition is quoted from \cite{Rey-Bellet2006b}. A simple and self-contained proof is also provided in this reference. For a more comprehensive discussion, we refer to \cite{dym2008gaussian}. \\

As detailed in \cite{Rey-Bellet2006b} the inverse density $\left ( \widehat{\rho}_{\K}(k) \right )^{-1}$ being a polynomial indeed implies the applicability of \cref{prop:markov:criteria:general}, as positivity of the measure $\widehat{\mu}_{\K}$ follows from Bochner's theorem. Therefore $\widehat{\rho}_{\K}$ must be a positive function, i.e., a positive polynomial of even degree, which in turn implies the existence of a suitable polynomial $p(k) =  \sum_{m=1}c_{m}(-ik)^{m}$ with properties as stated in \cref{prop:markov:criteria:general}.\\

\cref{prop:markov:criteria:general} has been used extensively to derive finite dimensional Markovian representations of the type of heat bath models used in \cite{Eckmann1998,Eckmann1999a,rey2003statistical,Rey-Bellet2006b}. Similarly, \cref{prop:markov:criteria:general}  can also be used 
to derive suitable distributions for the spring constants and the heat bath particle masses in the Ford-Kac model which ensure that in the thermodynamic limit the path of the distinguished particle converges weakly to the solution of a stochastic IDE which can be represented in a Markovian form; see \cite{Kupferman2002,Kupferman2004,Givon2004}.

\section{Ergodicity properties}\label{sec:ergodic:GLE} 
Let $e^{t\Lcgq}$ denote the associated evolution operator of the process \cref{eq:gle:q}, i.e.,
\begin{equation}
\forall \varphi  \in \mathcal{C}^{\infty}(\xDomain,\RR) : e^{t\Lcgq} \varphi(x) = \EE[ \varphi(\x(t)) \given \x(0) = x],
\end{equation}
where the expectation is taken with respect to the Brownian motion $\W$. In this section we derive criteria for exponential convergence of  $e^{t\Lcgq}$  in some weighted $L^{\infty}$ space as $t\rightarrow \infty$. More precisely, define for a prescribed $\Kc \in \mathcal{C}^{\infty}(\xDomain, [1,\infty))$ with the property that $\Kc(\x) \rightarrow \infty$ as $\norm{\x}\rightarrow \infty$ the set 
\begin{equation}
L^{\infty}_{\Kc}(\xDomain) := \left \{  \varphi \text{ measurable } :\norm{\varphi}_{L^{\infty}_{\Kc}}<\infty  \right \},
\end{equation}
where 
\begin{equation}
\norm{\varphi}_{L^{\infty}_{\Kc}} := \norm*{ \frac{\varphi}{\Kc}}_{\infty}, ~\varphi : \xDomain \rightarrow \RR~\text{measureable},
\end{equation}
so that $\big ( L^{\infty}_{\Kc}(\xDomain) , \norm{\cdot}_{L^{\infty}_{\Kc}} \big)$ can be verified to define a Banach space. Furthermore, denote by 
\begin{equation}
\EE_{\mu} \varphi  := \int \varphi(\x) \mu(\dd \x),
\end{equation}
the expectation of an observable $\varphi$ with respect to the probability measure $\mu$.\\

We show under certain conditions on the coefficients $\Gammaq,\Sigmaq$ and the force ${\bm F}$ that there exists a unique probability measure with smooth density $\mu(\dd \x)= \rho(\x)\dd \x$, such that


\begin{equation}\label{eq:geo:conv:gle}
\exists\, \kappa>0, C>0, ~\forall \,\varphi \in L^{\infty}_{\Kc}, ~
\norm*{ \EE_{\mu}\varphi - e^{t \Lcgq} \varphi }_{L^{\infty}_{\Kc}} \leq C e^{-\kappa t }\norm*{ \EE_{\mu}\varphi   - \varphi}_{L^{\infty}_{\Kc}},
\end{equation}
for all $t \geq 0$, and
\begin{equation}\label{eq:bound:moments}
\int_{\xDomain}\Kc(\x)\mu(\dd \x) <\infty,
\end{equation}
where $\Kc$ is a suitable Lyapunov function whose exact properties we specify below. In particular, if ${\bm F} = -\nabla U$ and \cref{as:gle:q} holds, then 
\[
\mu(\dd \x) = \mu_{\Q,\beta}(\dd \x),
\]
where $\mu_{\Q,\beta}$ is as defined in \cref{prop:GLE:cons:q}. If the process \cref{eq:gle:q} satisfies \cref{eq:geo:conv:gle} for all $t\geq 0$, we say in the sequel that it is {\em geometrically ergodic}.\\

All results are derived using standard Lyapunov techniques (see e.g.  \cite{meyn1993stability,Meyn1997,Mattingly2002,Rey-Bellet2006a}), which we summarize in \cref{sec:stochastic:analysis}. That is, we show that (i) the minorization condition (\cref{as:minorization:inf}) is satisfied and (ii) a suitable Lyapunov function exists which satisfies \cref{as:lyapunov:inf}  (or more generally the existence of a suitable class of Lyapunov functions of which each instance satisfies \cref{as:lyapunov:inf}). We treat the cases $\qDomain= \mathbb{T}^{n}$ and $\qDomain = \RR^{n}$ separately. In the situation $\qDomain = \RR^{n}$, we show geometric ergodicity for the case of constant coefficients, i.e.,  $\Gammaq \equiv \Gammabf$, and $\Sigmaq \equiv \Sigmabf$, which in the non-Markovian form \cref{eq:gle:q:nonmark}  corresponds to the situation of a stationary random force. For the case of a bounded domain $\qDomain = \mathbb{T}^{n}$ we can show geometric ergodicity also for the case where $\Gammaq$ and $\Sigmaq$ are not constant in $\q$, i.e., the random force, $\etaq$, in the corresponding non-Markovian form \cref{eq:gle:q:nonmark} is non-stationary. In order to simplify presentation we assume for the remainder of this article $\M = \I_{n}$. 
\subsection{Summary of main results}\label{sec:sum:results}
Let in the sequel $g(x) = \bT(f(x))$ indicate that the function $f$ is bounded both above and below by $g$ asymptotically as $\norm{x}\rightarrow \infty$, i.e., there exist $c_{1},c_{2}>0$ and $\tilde{x} \geq 0$, such that 
$c_{1}g(x) \leq f(x) \leq c_{2}g(x)$ for all $\norm{x} \geq \tilde{x}$.
\subsubsection{Results for stationary noise}
We first present results for the constant coefficient case, i.e., $\Gammaq\equiv \Gammabf$, and $\Sigmaq\equiv \Sigmabf$. Let for the remainder of this subsection $\Gammabf,\Sigmabf$ be such that 
\begin{enumerate}[label=(\roman*)]
\item $-\Gammabf$ is a stable matrix, i.e., the real parts of all eigenvalues of $\Gammabf$ are positive.
\item the SDE \cref{eq:gle:q} satisfies the parabolic H\"ormander condition both in the presence of the force term ${\bm F}$ as well as in absence of a force term, i.e., ${\bm F} \equiv 0$. We provide algebraic conditions on $\Gammabf,\Sigmabf$ which imply the parabolic H\"ormander condition in \cref{subsection:hypo:criteria}.
\item \cref{as:gle:q} is satisfied so that for ${\bm F} = -\nabla U$ the measure $\mu_{\Q,\beta}(\dd \x) = \rho_{\Q,\beta}(\x) \dd \x$ with $ \rho_{\Q,\beta}$ as defined in \cref{eq:invariant:measure} is an invariant measure of \cref{eq:gle:q}.
\end{enumerate}
\begin{theorem}\label{thm:ergo:bounded:1}
Let $\qDomain = \mathbb{T}^{n}$, and $\Gammaq,\Sigmaq$ as specified above. There is a unique invariant measure $\mu$ such that for any $l\in \mathbb{N}$ there exists $\Kc_{l} \in \mathcal{C}^{\infty}( \mathbb{T}^{n} \times \RR^{n+m} )$ with 
\[
\Kc_{l}(\q,\p,\s) = \bT(\norm{\z}^{2l}), ~ \text{ as } \norm{\z} \rightarrow \infty, ~\z= \begin{pmatrix}\p \\ \s \end{pmatrix},
\]
so that \eqref{eq:geo:conv:gle} and \eqref{eq:bound:moments} hold for $\Kc=\Kc_{l}$. In particular, if  ${\bm F} = -\nabla U$, then $\mu = \mu_{\Q,\beta}$.
\end{theorem}
\begin{proof}
The validity of the minorization condition follows from \cref{thm:minorization:torus}. The existence of a suitable class of Lyapunov functions is shown in \cref{prop:lya:torus}.
\qed\end{proof}
In the case of an unbounded configurational domain, i.e., $\qDomain=\RR^{n}$, we require an additional assumption on the force ${\bm F}$ in order to construct a suitable class of Lyapunov functions.

\begin{assumption}\label{as:potential:unbounded}
There exists a potential function $V \in \mathcal{C}^{2}(\qDomain, \RR)$ with the following properties
\begin{enumerate}[label=(\roman*)] 
\item \label{as:potential:it:0} there exists $G \in \RR$ such that  
\[
\inner{\q,{\bm F}(\q)}  \leq -  \inner{\q,\nabla_{\q} V(\q)} + G.
\]
for all $\q \in \qDomain$.
\item \label{as:potential:it:1} the potential function 
is bounded from below, i.e., there exists $u_{\min} > -\infty$ such that
\[
\forall \q \in \qDomain,~ V(\q) \geq u_{\min}.
\]
\item \label{as:potential:it:2}  there exist constants $D,E > 0 $ and $F \in \mathbb{R}$ 
such that
\begin{equation}\label{cond:potential}
\forall \q \in \qDomain,~\inner{ \q,  \nabla_{\q} V(\q)}\geq D V(\q) + E \| \q\|_{2}^{2} + F.
\end{equation}
\end{enumerate}
\end{assumption}


\begin{theorem}\label{thm:ergo:unbounded:2}
Let $\qDomain =\RR^{n}$, ${\bm F}$ satisfies \cref{as:potential:unbounded}, $\Gammaq,\Sigmaq$ as specified above with ${\rm rank}(\Sigmabf) = n+m$ and ${\rm rank}(\Gammabf_{1,1})=n$.
There is a unique invariant measure $\mu$ such that  for any $l\in \mathbb{N}$ there exists $\Kc_{l} \in \mathcal{C}^{\infty}(\RR^{2n+m}, [1,\infty) )$ with 
\[
\Kc_{l}(\x) = \bT(\norm{\x}^{2l}), ~ \text{ as } \norm{\x} \rightarrow \infty, 
\]
such that \eqref{eq:geo:conv:gle} and \eqref{eq:bound:moments} hold for $\Kc=\Kc_{l}$. In particular, if  ${\bm F} = -\nabla U$, then $\mu = \mu_{\Q,\beta}$.
\end{theorem}
\begin{proof}
The validity of a minorization condition follows from \cref{prop:gle:minorization:3}. The existence of a suitable class of Lyapunov functions is shown in \cref{prop:lya:R}.
\qed\end{proof}
The above theorem covers instances of the GLE with a non-degenerated white noise component. In order to derive geometric ergodicity for GLEs without a white noise component, i.e., $\Sigmabf_{1} = \0$ which is implied by $\Gammabf_{1,1}=\0$ (see \cref{lem:purecolor}), we require the force ${\bm F}$ to satisfy the following assumption:
\begin{assumption}\label{as:potential:unbounded:2}
Let the force ${\bm F}$ be such that
\[
{\bm F}(\q) = {\bm F}_{1}(\q) + {\bm F}_{2}(\q),
\]
where  ${\bm F}_{1}\in \mathcal{C}^{\infty}(\RR^{n},\RR^{n})$ is uniformly bounded in $\qDomain$, i.e., 
\[
\sup_{\q \in \qDomain}\norm{{\bm F}_{1}(\q)}_{\infty} < \infty, 
\]
and 
\[
{\bm F}_{2}(\q) ={\bm H} \q,
\]
with  ${\bm H}\in \RR^{n\times n}$ being a positive definite matrix, i.e.,  $\min \sigma({\bm H}) = \lambda_{{\bm H}}>0$. 
\end{assumption}
\begin{remark}\label{rem:on:as:potential:unbounded:2}
\cref{as:potential:unbounded:2} implies that there is $ \overline{H}>0$ and $ \overline{h}\in \RR$ so that
\[
\abs{ \langle {\bm g}, {\bm F}(\q)\rangle } \leq   \overline{H}\abs{\langle {\bm g},\q \rangle} + \overline{h},
\]
for all $\q,{\bm g} \in \RR^{n}$. Moreover, if both \cref{as:potential:unbounded:2}  and \cref{as:potential:unbounded} hold, then it is easy to see that the potential function $V$ in \cref{as:potential:unbounded} is of the form of a perturbed quadratic potential function in the following sense:
\[
V(\q) = V_{1}(\q) + V_{2}(\q),
\]
where  $V_{1}\in \mathcal{C}^{\infty}(\RR^{n},\RR)$ has bounded derivatives and 
\[
V_{2}(\q) =\frac{1}{2}\q^{\trans} {\bm H} \q.
\]
\end{remark}
The following theorem provides a sufficient condition for geometric ergodicity of \cref{eq:gle:q} for constant coefficients  and $\Gammabf_{1,1} = \0$.
\begin{theorem}\label{thm:ergo:unbounded:3} 
Let $\qDomain =\RR^{n}$, ${\bm F}$ satisfies \cref{as:potential:unbounded} and \cref{as:potential:unbounded:2}, 
and $\Gammaq, \Sigmaq$ 
as specified above with $\Gammabf_{1,1}=\0$. There exists a unique probability measure $\mu(\dd \x)$ such that for any $l\in \mathbb{N}$ there exists $\Kc_{l} \in \mathcal{C}^{\infty}(\RR^{2n+m}, [0,\infty))$ with 
\[
\Kc_{l}(\x) = \bT(\norm{\x}^{2l}), ~ \text{ as } \norm{\x} \rightarrow \infty, 
\]
such that \eqref{eq:geo:conv:gle} and \eqref{eq:bound:moments} hold for $\Kc=\Kc_{l}$. In particular, if ${\bm F} = -\nabla U$, then $\mu = \mu_{\Q,\beta}$.
\end{theorem}
\begin{proof}
The validity of the minorization condition follows from \cref{prop:minorization:R-2}. The existence of a suitable class of Lyapunov functions is shown in \cref{prop:lya:R}.
\qed\end{proof}
\subsubsection{Results for non-stationary noise}
For the case of a periodic configurational domain $\qDomain = \mathbb{T}^{n}$ we show geometric ergodicity for the SDE \cref{eq:gle:q} for the general case where $\Gammaq$ and $\Sigmaq$ may not be constant. We focus on the case 
\[
\Gammaq(\cdot) = 
\begin{pmatrix}
\0 &\Gammaq_{1,2}(\cdot)\\
\Gammaq_{2,1}(\cdot) & \Gammaq_{2,2}(\cdot)
\end{pmatrix} \in \mathcal{C}^{\infty}(\qDomain,\RR^{2n\times 2n}).
\]
where all non-vanishing sub-blocks are assumed to be invertible, i.e.,
 \[
 \Gammaq_{1,2}(\q),\Gammaq_{2,1}(\q), \Gammaq_{2,2}(\q),\Sigmaq_{2,2}(\q) \in {\rm GL}_{n}(\RR), 
 \]
 for all $\q\in \qDomain$, where by ${\rm GL}_{n}(\RR)\subset \RR^{n\times n}$ we denote the set of all invertible $n\times n$-matrices with real valued coefficients.
Furthermore, we assume that $-\Gammaq(\q)$ is a stable matrix for all $\q\in \qDomain$ and that $\Gammaq,\Sigmaq$ are such that \cref{as:gle:q} is satisfied, i.e., since $\Gammaq_{1,1} \equiv \0$, it follows by \cref{lem:purecolor} that
\begin{equation}\label{eq:pure:noise:general}
\forall \q \in \qDomain,~\Gammaq_{1,2}(\q)= -\Q \Gammaq_{2,1}(\q),
\end{equation}
holds. Moreover we assume 
\begin{equation}\label{as:C:matrix}
\exists \, \C \in \RR^{(n+m)\times (n+m)}~ \text{s.p.d.},~ \forall\, \q \in \qDomain : ~\Gammaq(\q)\C + \C\Gammaq^{\trans}(\q) ~\text{s.p.d.},
\end{equation}
where the notation ``s.p.d.'' stands for ``symmetric positive definite.''
We expect that our result can be easily extended to more general forms of $\Gammaq$, i.e., to the case where $\Gammaq(\q) \in \RR^{m\times m}$ with $m\neq n$; see note \ref{rem:kernel:generalization} at the end of this subsection. We also point out that the case $\Gammaq_{1,1}\neq \0$ would not cause any additional difficulties in the proof of the result as long as the identity \cref{eq:pure:noise:general} holds. (See e.g. \cite{sachs2017langevin} for ergodicity results for under-damped Langevin equation with non-constant coefficients.) 


\begin{theorem}\label{thm:ergo:GLEq}
Let $\qDomain=\mathbb{T}^{n}$. Under the assumptions on $\Gammaq$ and $\Sigmaq$ described in the preceding paragraph, there is a unique invariant measure $\mu$ such that there exists for any $l\in \mathbb{N}$ a function $\Kc_{l} \in \mathcal{C}^{\infty}(\mathbb{T}^{n} \times \RR^{2n}, [1,\infty))$ with 
\[
\Kc_{l}(\q,\p,\s) = \bT(\norm{\z}^{2l}), ~ \text{ as } \norm{\z} \rightarrow \infty, ~ \z= \begin{pmatrix}\p \\ \s \end{pmatrix},
\]
such that \eqref{eq:geo:conv:gle} and \eqref{eq:bound:moments} hold for $\Kc=\Kc_{l}$.  In particular, if  ${\bm F} = -\nabla U$, then $\mu = \mu_{\Q,\beta}$.
\end{theorem}
\begin{proof}
The validity of the minorization condition follows from \cref{thm:erg:gle:nonconst:coeff}. The existence of a suitable class of Lyapunov functions is shown in \cref{prop:lya:torus-2}.
\qed\end{proof}
We provide a simple example of an instance of \cref{eq:gle:q}, which satisfies the condition of \cref{thm:ergo:GLEq}: 
\begin{example}\label{exa:not:empty}
 Let  $m=n=1$ and let $\qDomain=\mathbb{T}$.  Consider the matrix-valued functions $ \Gammaq,\Sigmaq$ defined by 
 \[
 \Gammaq(\q) = \begin{pmatrix} 0 & -(2+\cos(2\pi \q))  \\ (2+\cos(2\pi \q)) & 1\end{pmatrix},~  \Sigmaq(\q) = \begin{pmatrix} 0 & 0 \\ 0 & 1\end{pmatrix}.
\]
Obviously, a valid choice for $\Q$ in \cref{prop:GLE:cons:q} is
\[
\Q = \begin{pmatrix} 0 & 0 \\ 0 & 1\end{pmatrix}.
\]
Moreover, 
\[
\C = \begin{pmatrix} 19/18 & -(1/6) \\ -(1/6) & 1\end{pmatrix}.
\]
satisfies \eqref{as:C:matrix}. This follows by virtue of \cref{lem:schur:posdef}. We provide a plot of the eigenvalues of the matrix 
\begin{equation}\label{eq:ex:R:matrix}
{\bm R}(\q) = \Gammaq(\q) \C+ \C  \Gammaq^{\trans}(\q),
\end{equation}
as a function of $\q$ in \cref{plot:eigenvalues}.
\begin{figure}[ht]
\begin{centering}
\hspace{-.5cm}
\includegraphics[width=.5\textwidth]{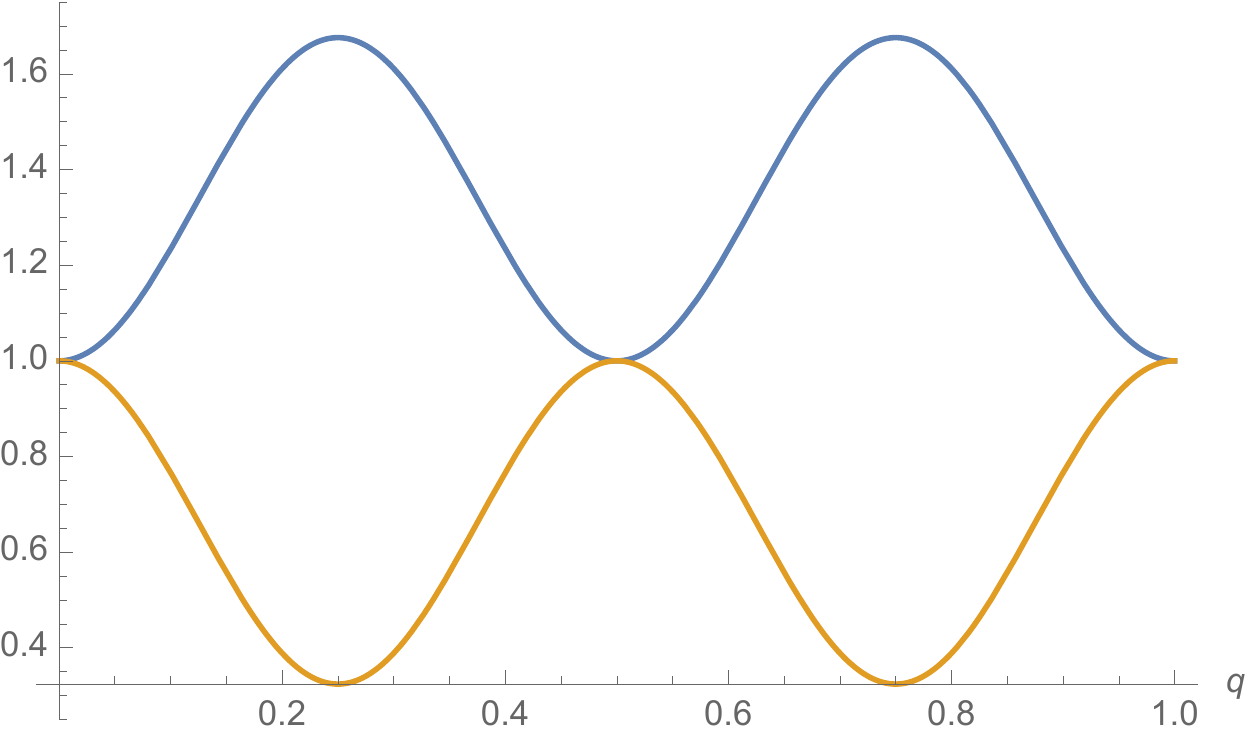}
\caption[]{$\q$ vs. the eigenvalues of the matrix ${\bm R}(\q)$ which is defined in \cref{eq:ex:R:matrix}. 
}\label{plot:eigenvalues}
\end{centering}
\end{figure}
\FloatBarrier
\end{example}
\subsubsection{Central limit theorem for quasi-Markovian GLE dynamics.}
A direct consequence of the geometric ergodicity of the dynamics \cref{eq:gle:q} is the validity of a central limit theorem for certain observables. This result is of practical importance as it justifies the use of GLE dynamics for sampling purposes as, e.g., in \cite{Ceriotti2009,Ceriotti2010,Wu2015}.

Define the projection operator 
\[
\Pi \varphi = \varphi -\EE_{\mu}\varphi,
\]
and let $L_{\Kc,0}^{\infty} := \Pi L_{\Kc}^{\infty} \subset L_{\Kc}^{\infty}$, be the subspace of $L_{\Kc}^{\infty}$ which is comprised of observables with vanishing mean. Denote by $\norm{\cdot}_{\opnorm(L_{\Kc}^{\infty})}$ the operator norm 
\[
\norm{A }_{\opnorm(L_{\Kc}^{\infty})} := \sup_{\varphi \in L_{\Kc}^{\infty}}\frac{ \norm{A \varphi }_{L_{\Kc}^{\infty}}}{\norm{\varphi }_{L_{\Kc}^{\infty}}}.
\]
induced by the norm $\norm{\cdot }_{L_{\Kc}^{\infty}}$ for operators $A : L_{\Kc}^{\infty} \rightarrow L_{\Kc}^{\infty}$. The validity of \cref{eq:geo:conv:gle} for all $t\geq0$ immediately implies the inequality
\begin{equation}
\norm{\Pi e^{t\Lcgq}}_{\opnorm(L_{\Kc}^{\infty})} \leq C e^{t\kappa}.
\end{equation}
By \cite[Proposition 2.1]{Lelievre2016a}, $\Lcgq$ considered as an operator on $ L_{\Kc,0}^{\infty}$ is invertible with bounded spectrum. By \cite{Bhattacharya1982} this implies a central limit theorem for observables contained in $\varphi \in L_{\Kc}^{\infty}$ as summarized in the following \cref{cor:CLT}.

\begin{corollary}\label{cor:CLT}
Let the conditions of one of the \cref{thm:ergo:bounded:1,thm:ergo:unbounded:2,thm:ergo:unbounded:3,thm:ergo:GLEq} be satisfied and let $\Kc_{l}$ for $l\in \NN$ be a suitable Lyapunov function as specified therein. The spectrum of $\Lcgq^{-1}\Pi$ is bounded in $\norm{\cdot}_{\opnorm(L_{\Kc_{l}}^{\infty})}$, i.e.,
\begin{equation}
\norm{\Lcgq^{-1}\Pi }_{\opnorm(L_{\Kc_{l}}^{\infty})} \leq \frac{C_{l}}{\kappa_{l}},
\end{equation}
where $C_{l},\kappa_{l}>0$ are such that \cref{eq:geo:conv:gle} holds for $\Kc=\Kc_{l}, \kappa=\kappa_{l}, C = C_{l}$. In particular, a central limit theorem  holds for the solution of \cref{eq:gle:q}, i.e., 
\begin{equation}
T^{-1/2} \int_{0}^{T} \left [ \EE_{\mu}\varphi - \varphi(\x(t))\right ] \dd t \sim \mathcal{N}(0,\sigma_{\varphi}^{2}), \text{ as } T\rightarrow \infty,
\end{equation}
for any $\varphi \in L_{\Kc_{l}}^{\infty}$, where $\mu$ denotes the unique invariant measure of $\x$ and 
\[
\sigma_{\varphi}^{2} = -2 \int \left ( \Lcgq^{-1} \Pi \varphi(\x) \right) \Pi\varphi(\x)\mu(\dd \x).
\]
\end{corollary}

\subsubsection{Notes on \cref{thm:ergo:bounded:1,thm:ergo:unbounded:2,thm:ergo:unbounded:3,thm:ergo:GLEq}:}\label{sec:notes}
\begin{enumerate}[label=\textbf{N.\arabic*}]
\item \cref{thm:ergo:bounded:1,thm:ergo:unbounded:2,thm:ergo:unbounded:3,thm:ergo:GLEq} imply path-wise ergodicity in the sense that
\begin{equation}
\lim_{T \rightarrow \infty} \frac{1}{T} \int_{0}^{T}\varphi(\x(t)) \dd t = \EE_{\mu}\varphi, 
\end{equation}
almost surely for $\mu$-almost all initializations of $\x(0)$ an almost all realizations of the Wiener process $\W$. We note that in the case that ${\bm F} = -\nabla U$ and \cref{as:gle:q} is satisfied it is sufficient to show that the generator $\Lcgq$ is hypoelliptic in order to conclude uniqueness of the invariant measure and path-wise ergodicity in the above sense. This follows directly from the arguments in \cite{kliemann1987recurrence} as in this case the form of the invariant measure is known and has a smooth positive density.\\
\item
The Lyapunov-based techniques  on which the proofs of our ergodicity results rely have been studied in the context of stochastic differential equations (see \cite{meyn1993stability,Talay2002,Mattingly2002,Rey-Bellet2006a}) as well as in the context of discrete time Markov chains  (see e.g. \cite{harris1956existence,meyn1992stability,Meyn1997,Hairer2011a}). In particular, we mention the application of these techniques to prove geometric ergodicity of solutions of the under-damped Langevin equation in \cite{Talay2002,Mattingly2002,Rey-Bellet2006a}. As discussed in \cref{subsec:generator},
 the structure of the SDE \cref{eq:gle:q} resembles the structure of the under-damped Langevin equation and it is therefore not surprising that also
the structure of the Lyapunov functions constructed in the proofs of \cite{Mattingly2002} resemble the structure of the Lyapunov functions presented in the latter two references. \\
\item In \cite{Ottobre2011} the authors construct a Lyapunov function for a Markovian reformulation of the GLE with conservative force which in the representation  \cref{eq:gle:q} corresponds to the case where $\Gammaq,\Sigmaq$ are constant with $\Gammaq \equiv \Gammabf$ such that $\Gammabf_{1,1}=\0$ and $\Gammabf_{1,2}, \Gammabf_{2,1},\Gammabf_{2,2}\in \mathbb{R}^{n\times n}$ are diagonal matrices. In the same article  exponential convergence of the law to a unique invariant distribution $\mu$ in relative entropy is shown and exponential decay estimates for the semi-group operator $e^{t\Lcgq}$ in weighted Sobolev space $H^{1}(\mu)$ are derived using the hypocoercivity framework by Villani (see \cite{Villani2009}). \\
\item Ergodic properties of non-equilibrium systems which have a similar structure as the QGLE models considered here have been studied in a series of papers \cite{Eckmann1998,Eckmann1999a,Eckmann1999,Rey-Bellet2002,rey2003statistical}.  These systems consist of a chain of a finite number of oscillators whose ends are coupled to two different heat baths. In a simplified version these systems can be written in the form
\begin{equation}\label{eq:NEchain}
\begin{aligned}
\dot{\bm r}_{1} &= - \gamma_{1} {\bm r}_{1} + \lambda_{1} \p_{1} + \sqrt{2 \beta^{-1} \gamma_{1}} \dot{W}_{1},\\
\dot{\q}_{1} &= \p_{1},\\
\dot{\p}_{1} &= - \partial_{\q_{1}} U(\q) - \lambda_{1} {\bm r}_{1},\\
\dot{\q}_{i} &= \p_{i}, &i=2,3,\dots, n-1,\\
\dot{\p}_{i} &= -\partial_{\q_{i}}U(\q),~&i=2,3,\dots, n-1,\\
\dot{\q}_{n} &= \p_{n},\\
\dot{\p}_{n} &= - \partial_{\q_{n}} U(\q) - \lambda_{2} {\bm r}_{2},\\
\dot{\bm r}_{2} &= - \gamma_{2} {\bm r}_{2} + \lambda_{2} \p_{n} + \sqrt{2 \beta^{-1} \gamma_{2}} \dot{W}_{2},
\end{aligned}
\end{equation}
where
\[
U(\q) = U_{1}(\q_{1}) + U_{n}(\q_{n}) + \sum_{i=2}^{n}\tilde{U}(\q_{i} - \q_{i-1}),
\]
with $U_{1},U_{2},\tilde{U} \in \mathcal{C}^{\infty}(\RR,\RR)$, $\gamma_{i}>0,\lambda_{i}>0$ for $i=1,2$, and $W_{1},W_{2}$ are two independent Wiener processes taking values in $\RR$.   
Under certain conditions on the potential functions $U_{1},U_{n}$ and $\tilde{U}$, the existence of an invariant measure (stationary non-equilibrium state) has been shown in \cite{Eckmann1998}. Uniqueness conditions were derived in \cite{Eckmann1999a,Eckmann1999}, and exponential convergence to the invariant state was shown in  \cite{Rey-Bellet2002} (see also the review paper \cite{Rey-Bellet2006b} and \cite{carmona2007existence}). In the latter reference slightly more general heat bath models are considered than above in \cref{eq:NEchain}). Exponential convergence towards a unique invariant measure is proven in \cite{Rey-Bellet2002}  by showing the existence of a suitable Lyapunov function and by showing hypoellipticity and controllability in the sense of \cref{as:control}. The construction of a suitable control in the proof provided therein relies on $\tilde{U}$ being strictly convex. We expect that the techniques which are used in \cite{Rey-Bellet2002} to prove the existence of a suitable Lyapunov function and the controllability of the SDE can be extended/modified to prove geometric ergodicity for a wide range of GLEs which can be represented in the form \cref{eq:gle:q} with constant coefficients. In fact it has been demonstrated in \cite{Rey-Bellet2006b} that controllability  in the sense of \cref{as:control}  of a system consisting of a chain of oscillators which are coupled to a single heat bath, can be proven by the same techniques as used in \cite{Rey-Bellet2002}.\\
\item \label{rem:kernel:generalization}
We expect that \cref{thm:ergo:GLEq} can be generalized to cover instances of \eqref{eq:gle:q}, where $\Gammaq$ is of a form such that in the non-Markovian reformulation \cref{eq:gle:q:nonmark} the memory kernel is of the form
\[
\K_{\Gammaq}(\q,t) =  \Gammaq_{1,1}(\q)\delta(t) - \sum_{i=1}^{K} \Gammaq_{1,2}^{(i)}(\q) e^{-t\Gammabf_{2,2}^{(i)}}\Gammaq_{2,1}^{(i)}(\q), ~K \in \mathbb{N},
\]
where each $\Gammaq^{(i)}$,
\[
\Gammaq^{(i)}(\q)=\begin{pmatrix}
\0 &\Gammaq_{1,2}^{(i)}(\q)\\
\Gammaq_{2,1}^{(i)}(\q) & \Gammabf_{2,2}^{(i)}
\end{pmatrix}
\]
satisfies the same conditions as $\Gammaq$ in \cref{thm:ergo:GLEq}. 
\end{enumerate}
\subsection{Conditions for hypoellipticity}\label{subsection:hypo:criteria}
Consider the case of constant coefficients in \cref{eq:gle:q}, i.e., $\Gammaq \equiv \Gammabf, \Sigmaq \equiv \Sigmabf$. In this subsection we provide criteria in the form of algebraic conditions on $\Gammabf$ and $\Sigmabf$ which ensure that \cref{eq:gle:q} satisfies the parabolic H\"ormander condition, which by \cref{thm:hormander}, implies that the differential operators 
\[
\Lcgq,~\Lcgq^{\dagger},~\partial_{t} - \Lcgq,~\partial_{t} - \Lcgq^{\dagger},
\]
are hypoelliptic. 
Let in the following \cref{prop:cond:hyp} $\Sigma_{i}, 1\leq i \leq n+m$ denote the column vectors of $\Sigmabf$, i.e.,
\[
\Sigmabf = \left [\Sigma_{1}, \dots \Sigma_{m+n} \right ] \in \RR^{(n+m)\times(n+m)}.
\]
\begin{proposition}\label{prop:cond:hyp}
Let $\Gammaq \equiv \Gammabf \in \RR^{(n+m)\times(n+m)}$ such that $-\Gammabf$ is stable and $\Sigmaq \equiv\Sigmabf \in \RR^{(n+m)\times(n+m)}$. Any of the following conditions
is sufficient for \cref{eq:gle:q} to satisfy the parabolic H\"ormander condition.
\begin{enumerate}[label=(\roman*)]
\item \label{item:cond:hyp:1} ${\bm F} = {\bm H} \q +{\bm h} $, where ${\bm H} \in \RR^{n\times n}, {\bm h} \in \RR$,  and for all $\q\in \qDomain $
\begin{equation}\label{item:cond:hyp:1:eq:1}
\mathbb{R}^{2n+m} = {\rm lin}\left( \left \{ {\bm S}^{k} \begin{pmatrix} \0 \\ \Sigma_{i} \end{pmatrix}: ~k \in \mathbb{N},~ 1 \leq i \leq n+m \right \} \right ), 
\end{equation}
where
\[
{\bm S} := 
-\begin{pmatrix}
\0 & -\I_{n} & {\bm 0}\\
{\bm H} &  \Gammabf_{1,1} &\Gammabf_{1,2}\\
{\bm 0} & \Gammabf_{2,1}  &\Gammabf_{2,2}
\end{pmatrix}\in \RR^{(2n+m)\times(2n+m)}.
\]
\item \label{item:cond:hyp:2}
\begin{equation}\label{item:cond:hyp:2:eq:1}
\mathbb{R}^{n+m} = {\rm lin}\left(  \bigcup_{1\leq i \leq n+m}\left \{ \Gammabf^{k}\Sigma_{i} : ~k \leq k_{i} \right \} \right ), 
\end{equation}
where $k_{i},  ~1\leq i \leq n+m$ are defined as 
\begin{equation}\label{item:cond:hyp:2:eq:2}
k_{i} := \argmax_{k \in \mathbb{N}} {\bm S}_{0}^{k} \begin{pmatrix} \0 \\ \Sigma_{i}  \end{pmatrix}\in \{ \0\} \times \RR^{n+m},
\end{equation}
with
\[
{\bm S}_{0} := 
-\begin{pmatrix}
\0 & -\I_{n} & {\bm 0}\\
{\bm 0} &  \Gammabf_{1,1} &\Gammabf_{1,2}\\
{\bm 0} & \Gammabf_{2,1}  &\Gammabf_{2,2}
\end{pmatrix} \in \RR^{(2n+m)\times(2n+m)}.
\]
\item\label{item:cond:hyp:3} 
 ${\rm rank}\left ( \Sigmabf_{2} \right ) = m $, and ${\rm rank}\left ( \Gammabf_{1,2}\right ) =n$.
\end{enumerate}
\end{proposition}
\begin{proof}
In relation to \cref{thm:hormander} the coefficients $\diffusion_{i}$ are  
\[\diffusion_{0}(\x) =
-\G \begin{pmatrix}
- {\bm F}(\q)\\
\z
\end{pmatrix},
\]
and
\[
\diffusion_{i} = 
\beta^{-\frac{1}{2}}
\begin{pmatrix}
\0 \\
\Sigma_{i}
\end{pmatrix} \in \mathbb{R}^{2n+m},  1 \leq i \leq n+m,
\]
with 
$\G\in \RR^{(2n+m)\times(2n+m)}$ as defined in \eqref{eq:def:G:matrix}. 
Since for $i>0$ the coefficients $\diffusion_{i}$ are constant in $\x$, we find 
$[\diffusion_{i},\diffusion_{j}] = \0$ and $[\diffusion_{0},\diffusion_{i}] = - \nabla_{\x} \diffusion_{0}\,\diffusion_{i}$ for $i,j > 0$, where $\nabla_{\x} \diffusion_{0}$ denotes the Jacobian matrix of  $\diffusion_{0}$, i.e.,
\[
\nabla_{\x} \diffusion_{0} = 
-\begin{pmatrix}
\0 & -\I_{n} & {\bm 0}\\
-\nabla {\bm F}(\q) &  \Gammabf_{1,1} &\Gammabf_{1,2}\\
{\bm 0} & \Gammabf_{2,1}  &\Gammabf_{2,2}
\end{pmatrix},
\] 
and $\nabla {\bm F}(\q)$ denotes the Jacobian of the force ${\bm F}$. Therefore,
\begin{equation}\label{rec:hoermander}
\mathscr{V}_{1} = \{ -\nabla_{\x} \diffusion_{0} \,{\bm v} : {\bm v} \in  \mathscr{V}_{0} \} \cup  \mathscr{V}_{0},
\end{equation}
where 
\[
 \left \{ \begin{pmatrix}
\0 \\
\Sigma_{i}
\end{pmatrix}  \right   \}_{i=1}^{n+m},
\]
\begin{itemize}
\item In the case of \ref{item:cond:hyp:1} it follows that $
\nabla_{\x} \diffusion_{0}(\x)  = {\bm S}.$ In particular, since $\nabla_{\x} \diffusion_{0}$ is constant in $\x$, \eqref{rec:hoermander} generalizes to
\begin{equation}\label{rec:hoermander:2}
\mathscr{V}_{i+1} =  \{  {\bm S} \,{\bm v} : {\bm v} \in  \mathscr{V}_{i} \} \cup  \mathscr{V}_{i}, ~i \in \NN.
\end{equation}
Since $\mathscr{V}_{i}$ consists only of constant functions, we have ${\rm lin}(\mathscr{V}_{i}(\x)) \equiv {\rm lin}(\mathscr{V}_{i})$ for all $\x\in \xDomain, i \in \mathbb{N}$, thus \eqref{rec:hoermander:2} implies that \eqref{item:cond:hyp:1:eq:1} is a sufficient condition for the SDE \eqref{eq:gle:q} to satisfy the parabolic H\"ormander condition.
\item Regarding \ref{item:cond:hyp:2}: Let $k_{\max} = \max_{1\leq i \leq n+m } k_{i}$. $k_{i}$ being as defined in \eqref{item:cond:hyp:2:eq:2} together with \eqref{item:cond:hyp:2:eq:1} ensures that there is $\widetilde{\mathscr{V}} \subset \mathscr{V}_{k_{\max}} $ such that all elements in $\widetilde{\mathscr{V}}$ are constant and 
\[
{\rm lin} \left (  \widetilde{\mathscr{V}} \right ) \equiv \begin{pmatrix} \0 \\ \RR^{n+m} \end{pmatrix}.
\]
Therefore, 
\begin{equation}\label{rec:hoermander:3}
\mathscr{V}_{k_{\max}+1} \supset  \{  -\nabla \diffusion_{0} \,{\bm v} : {\bm v} \in  \widetilde{\mathscr{V}} \} \cup   \widetilde{\mathscr{V}},
\end{equation}
thus for all $\x\in \xDomain$
\[
{\rm lin} \(\mathscr{V}_{k_{\max}+1}(\x) \) =  {\rm lin} \left ( \{  -\nabla \diffusion_{0}(\x) \,{\bm v}(\x) : {\bm v} \in  \widetilde{\mathscr{V}} \} \cup   \widetilde{\mathscr{V}}(\x) \right ) = \RR^{2n+m},
\]
where the latter equivalence is due to the fact that
\[
{\rm lin} \left ( \{  -\nabla \diffusion_{0}(\x) {\bm v} \, : {\bm v} \in B \} \cup   B  \right ) =\RR^{2n+m}, 
\]
for all $\x\in\xDomain$ and any basis $B \subset \RR^{2n+m}$ of $\{\0\} \times \RR^{n+m}$.
\item Regarding \ref{item:cond:hyp:3}: Since ${\rm lin}(\Sigmabf_{2})= \RR^{m}$ and ${\rm rank} \left ( \Gammabf_{1,2}\right )= n$ it follows that 
\[
\{\0\} \times  \mathbb{R}^{n+m} = {\rm lin}\left( \left \{ \begin{pmatrix} \0 \\ \Gammabf\Sigma_{i} \end{pmatrix}, 1 \leq i \leq n+m  \right \} \right ),
\]
thus the result follows by \ref{item:cond:hyp:2}.
\end{itemize}
\qed\end{proof}

\subsection{Technical lemmas required in the proofs of ergodicity of \eqref{eq:gle:q} with stationary random force}\label{sec:lem:stat}
In this subsection we provide the necessary technical lemmas to which we refer in the proofs of \cref{thm:ergo:bounded:1,thm:ergo:unbounded:2,thm:ergo:unbounded:3}, thus in the remainder of this subsection we assume $\Gammaq \equiv \Gammabf, \Sigmaq \equiv \Sigmabf$. 
We begin by showing the existence of a class of suitable Lyapunov functions in the case of a bounded configurational domain, i.e., $\qDomain=\mathbb{T}^{n}$.
\begin{lemma}\label{prop:lya:torus}
Let $\qDomain = \mathbb{T}^{n}$, $-\Gammabf \in \RR^{(n+m)\times(n+m)}$ stable, 
then
\[
\mathcal{K}_{l}(\q,\p,\s) = \left ( \z^{\trans} \C \z \right )^{l}+1, ~ l \in \mathbb{N},
\]
where $\C \in \RR^{n+m}$ is a symmetric positive definite matrix such that  $\Gammabf^{\trans} \C + \C \Gammabf$ is positive definite, 
defines a family of Lyapunov functions for the differential operator $\Lcgq$, i.e.,  for each $l \in \mathbb{N}$ there exist constants $a_{l}>0$, $b_{l}\in \RR$, 
such that for $\Lc = \Lcgq, \Kc = \Kc_{l}$, \cref{as:lyapunov:inf} holds with $a = a_{l}, b=b_{l}$.
\end{lemma}
\begin{proof}
We show the existence of suitable constants $\widetilde{a}_{l},\widetilde{b}_{l}$ so that the inequality  \cref{eq:lyapunov:inf} is satisfied for $\Kc = \widetilde{\Kc}_{l} :=  \Kc_{l}-1$, and $\Lc = \Lcgq,  a=\widetilde{a}_{l}, b = \widetilde{b}_{l}$, which directly implies the statement of \cref{prop:lya:torus} for $a_{l} = \widetilde{a}_{l}$ and $b_{l} = \widetilde{b}_{l} +\widetilde{a}_{l}$. 
$-\Gammabf$ being a stable matrix ensures that there indeed exists a symmetric positive definite matrix $\C$ such that $\Gammabf^{\trans} \C + \C \Gammabf$ is positive definite. 
Without loss of generality let $\min \sigma (\C) = 1$, so that 
\begin{equation}\label{eq:prop:lya:torus:ieq:1}
\norm{\z}_{2}^{2} \leq \z^{\trans}\C \z = \Kc_{1}(\x) -1.
\end{equation}
Furthermore,
\[
\lambda = \sup_{\z \in \zDomain, \norm{\z}_{2}=1} \frac{ \z^{\trans}(\Gammabf^{\trans}\C + \C \Gammabf) \z}{ \z^{\trans} \C\z},
\]
so that
\begin{equation}\label{eq:prop:lya:torus:ieq:2}
 2\z^{\trans}\Gamma^{\trans}\C\z  \geq  \lambda \z^{\trans}\C \z  =\lambda ( \Kc_{1}(\x) - 1).
\end{equation}

We first consider the case $l=1$:
\begin{equation*}
\begin{aligned}
( \mathcal{L}_{H}+\mathcal{L}_{O}) \widetilde{\Kc}_{1}(\x) 
=& \,
 [ 2\p^{\trans}\C_{1,1} + 2 \s^{\trans} \C_{1,2} -  \p^{\trans}  ] {\bm F}(\q) - 2 \z^{\trans} \Gammabf^{\trans}\C \z \\
 &+ \beta^{-1}\sum_{i,j}\left[\C\Sigmabf \Sigmabf^{\trans}\C \right]_{i,j} \\
\leq & ~
 c_1 ||\z||_{2}- 2 \z^{\trans} \Gammabf^{\trans}\C \z + \beta^{-1}\sum_{i,j}\left[\C\Sigmabf \Sigmabf^{\trans}\C \right]_{i,j} \\
 \leq & ~
  \frac{c_1}{\epsilon_{1}} + \epsilon_{1} ||\z||_{2}^{2}- 2 \z^{\trans} \Gammabf^{\trans}\C \z + \beta^{-1}\sum_{i,j}\left[\C\Sigmabf \Sigmabf^{\trans}\C \right]_{i,j}, 
\end{aligned}
\end{equation*}
where 
\[
c_1 =  \max_{\q \in \qDomain, \|\z\|_{2} \leq 1} [ 2\p^{\trans}\C_{1,1} + 2 \s^{\trans} \C_{1,2} -  \p^{\trans}  ]  {\bm F}(\q).
\]
Thus, by \cref{eq:prop:lya:torus:ieq:1} and \cref{eq:prop:lya:torus:ieq:2},
\begin{equation*}
\begin{aligned}
( \mathcal{L}_{H}+\mathcal{L}_{O}) \widetilde{\Kc}_{1}(\x) &\leq  \frac{c_1}{\epsilon_{1}} + \epsilon_{1}\widetilde{\Kc}_{1}(\x)  -\lambda \widetilde{\Kc}_{1}(\x)  + \beta^{-1}\sum_{i,j}\left[\C\Sigmabf \Sigmabf^{\trans}\C \right]_{i,j}   \\
&=
- \widetilde{a}_{1}\widetilde{\Kc}_{1}(\x) + \widetilde{b}_{1},
\end{aligned}
\end{equation*}
with 
\[
\widetilde{a}_{1} := (\lambda-\epsilon_{1}), ~
\widetilde{b}_{1} := \frac{c_1}{\epsilon_{1}}  +  (\lambda+ \epsilon_{1}) + \beta^{-1}\sum_{i,j}\left[\C\Sigmabf \Sigmabf^{\trans}\C \right]_{i,j},
\]
so that  $\widetilde{a}_{1}>0$ for sufficiently small $\epsilon_{1} >0 $.\\

For $l>1$ we find:
\begin{equation}
\begin{aligned}
(\mathcal{L}_{H} + \mathcal{L}_{O})\widetilde{\Kc}_{l}(\x) 
=& ~
 l \widetilde{\Kc}_{l-1}(\x) \big [ \mathcal{L}_{H}\widetilde{\Kc}_{1}(\x) 
+ (-(\Gammabf\z) \cdot \nabla_{\z} \widetilde{\Kc}_{1}(\x)) 
+ \beta^{-1}\sum_{i,j}\left[\Sigmabf \Sigmabf^{\trans}\C \right]_{i,j} \big]\\
&+ 2 l (l-1) \beta^{-1}\z^{\trans}\C\Sigmabf\Sigmabf^{\trans} \C \z \widetilde{\Kc}_{l-2}(\x).
\end{aligned}
\end{equation}
Let
\[
\widetilde{\lambda} := \sup_{\x \in \xDomain, \norm{\z}_{2}=1} \left ( \frac{\z^{\trans}\C\Sigmabf\Sigmabf^{\trans} \C \z}{\widetilde{\Kc}_{1}(\x)} \right ),
\]
so that
\[
\forall \x \in \xDomain,~\z^{\trans}\C\Sigmabf\Sigmabf^{\trans} \C \z \widetilde{\Kc}_{l-2}(\x) \leq  \widetilde{\lambda}  \widetilde{\Kc}_{l-1}(\x).
\]
Thus, with
\[
c_{l} := \min \left ( 0, - \beta^{-1}\sum_{i,j}\left[\C\Sigmabf \Sigmabf^{\trans}\C \right]_{i,j} +  \beta^{-1}\sum_{i,j}\left[\Sigmabf \Sigmabf^{\trans}\C \right]_{i,j}  + 2(l-1)\beta^{-1}\widetilde{ \lambda} \right ),
\]
we find
 \begin{equation}
\begin{aligned}
(\mathcal{L}_{H} + \mathcal{L}_{O})\widetilde{\Kc}_{l}(\x)
\leq & ~
  l \widetilde{\Kc}_{l-1}(\x) \left( (\Lc_{H}+\Lc_{O}) \widetilde{\Kc}_{1}(\x)  + c_{l} \right )\\  
\leq & ~
l \widetilde{\Kc}_{l-1}(\x) \left (  - \widetilde{a}_{1} \widetilde{\Kc}_{1}(\x) + \widetilde{b}_{1} + c_{l} \right ) \\
\leq & ~
l \left ( -\widetilde{a}_{1}  \Kc_{l}(\x) + \frac{\widetilde{b}_{1}+ c_{l} }{\epsilon_{l}^{l-1}}  + \epsilon_{l}\Kc_{l}(\x) \right )
=  -\widetilde{a}_{l} \Kc_{l}(\x) + \widetilde{b}_{l},
\end{aligned}
\end{equation}
with
\[
\widetilde{a}_{l} := l(\widetilde{a}_{1}-\epsilon_{l}), ~
\widetilde{b}_{l} :=   l \frac{\widetilde{b}_{1}+c_{l}}{\epsilon_{l}^{l-1}},
\]
where $\epsilon_{l} > 0$ is chosen sufficiently small so that $\widetilde{a}_{l}>0$.
\qed\end{proof}

We next show the existence of a minorization condition in the case of $\qDomain=\mathbb{T}^{n}$. The idea of the proof is to decompose the diffusion process into an Ornstein-Uhlenbeck process and a bounded remainder term, which then enables us to conclude the existence of a minorizing measure by virtue of the fact that the solution of Fokker-Planck equation associated with the Ornstein-Uhlenbeck process is a non-degenerate Gaussian at all times $t>0$ and thus has full support. The idea of this approach is borrowed from \cite{LeMaSt2015} where it was used to show the minorization condition for a discretized version of the under-damped Langevin equation. Other applications of this trick can be found in \cite{redon2016error,Joubaud2014}.

\begin{lemma}\label{thm:minorization:torus}
Let $\qDomain=\mathbb{T}^{n}$. If  $\Gammabf\in \RR^{(n+m)\times(n+m)}$ and $\Sigmabf \in \RR^{(n+m)\times(n+m)}$ are as in \cref{thm:ergo:bounded:1}, then \cref{as:minorization:inf} (minorization condition) holds for the SDE \eqref{eq:gle:q}.
\end{lemma}
\begin{proof}[Proof of \cref{thm:minorization:torus}]
Let $\q(0) = \q_{0}$ and $\z(0) = \z_{0}$ with 
\[
(\q_{0},\z_{0}) \in \qDomain \times \mathcal{C}_{r},
\]
where
\[
C_{r} = \{ z \in \zDomain : \norm{z} < r \},
\] 
for arbitrary but fixed $r>0$. 

We can write the solution of \cref{eq:gle:q} as
\begin{equation}
\z(t) = \z_{0} +  \zDet(t) + \zStoch(t), ~~ \q(t) =\q_{0} +  \qDet(t) + \qStoch(t),
\end{equation}
with 
\[
\zDet(t) = \int_{0}^{t}  e^{-(t-s)\Gammabf}\begin{pmatrix} {\bm F}(\q(s)) \\ \0 \end{pmatrix}\dd s, ~~ \zStoch(t) = \int_{0}^{t} e^{-(t-s)\Gammabf} \Sigmabf \dd \W(s),
\]
and
\[
\qDet(t) = \int_{0}^{t} \Pi_{\p} \zDet(s) \dd s, 
~~\qStoch(t) = \int_{0}^{t} \Pi_{\p} \zStoch(s) \dd s.
\]
\\
The variables $\qStoch(t)$ and $\zStoch(t)$ are correlated and Gaussian, i.e.,
\[
\begin{pmatrix} \qStoch(t) \\ \zStoch(t) \end{pmatrix} \sim \mathcal{N} ( {\bm \mu}_{t}, \mathcal{V}_{t} ),
\]
with some $ {\bm \mu}_{t} \in \xDomain$ and $\mathcal{V}_{t} \in \RR^{(2n+m) \times (2n+m)}$. More specifically, $\tilde{\z}(t) = \z(0)+\zStoch(t)$ and $\q(0) +\qStoch(t)$ corresponds to the solution of 
of the linear SDE
\begin{equation}\label{gle:free}
\begin{aligned}
\dot{\tilde{\q}} &= \tilde{\p}, \\\
\dot{\tilde{\z}} & = - \Gammabf \tilde{\z} + \Sigmabf \dot{\W},\\
\end{aligned}
\end{equation}
where $\tilde{\z}(t) = (\tilde{\p}(t),\tilde{\s}(t))\in \pDomain \times \sDomain$. The law of $\tilde{\q}(t),\tilde{\z}(t)$ has full support for all $t>0$, provided that the covariance matrix $\mathscr{V}_{t}$ is invertible. 
This is indeed the case since $\Gammabf$ and $\Sigmabf$ are required to be such that \cref{eq:gle:q} satisfies the  parabolic H\"ormander condition. It follows that the system \eqref{gle:free} satisfies the parabolic H\"ormander condition. By \cref{thm:hormander}, we conclude that the law of $(\tilde{\q}(t),\tilde{\z}(t))$ has a density with respect to the Lebesgue measure for any $t>0$, which rules out the possibility of $\mathscr{V}_{t}$ being singular.\\

Let $\C\in\RR^{(n+m)\times(n+m)}$ be symmetric positive definite such that $\Gammabf \C + \C \Gammabf^{\trans}$ is positive definite as well, and consider the norm $\norm{\cdot}_{\C}$,
\[
\norm{\cdot}_{\C}:=\z^{\trans}\C\z, ~\z \in \RR^{n+m}.
\] 
The increment $\zDet(t)$ is uniformly bounded since 
\[
\norm{ \zDet(t) }_{\C} \leq \norm{\Gammabf^{-1}}_{\mathcal{B}(\C)} \norm{{\bm F}}_{L^{\infty}} <\infty,
\]
where 
\[
\norm{\Gammabf^{-1}}_{\mathcal{B}(\C)} := \max_{v \in \RR^{2n}} \frac{\norm{\Gammabf^{-1}v}_{\C}}{\norm{v}_{\C}} = \frac{1}{2}\min \sigma\( \Gammabf^{\trans}\C +\C\Gammabf \),
\]
denotes the operator norm of $\Gammabf^{-1}$ induced by $\norm{\cdot}_{\C}$.
It follows that also $\qDet(t)$ is bounded since
\[
\norm{\qDet(t)} \leq t\norm{ \zDet(t) }_{\C}  < \infty.
\]
\\
Let $\mu_{x_{0},t}$ denote the law of  $(\q(t),\z(t))$ and $\rho_{x_{0},t}$ be the associated density. For fixed $t>0$, the terms $\qDet(t)$ and $\zDet(t)$ are bounded and the law of $(\q(0)+\qStoch(t), \z(0)+\zStoch(t))$ has full support, in particular the measure $\mu_{x_{0},t}(\dd \x) = \rho_{x_{0},t}(\x)\dd \x $ of the superposition 
\[
(\q(t),\z(t)) =(\q(0)+\qDet(t)+\qStoch(t), \z(0)+\zDet(t)+\zStoch(t))
\]
has full support.
Now define $\rho \in \mathcal{C}(\xDomain, \RR_{+})$ as 
\[
\rho(x) := \min_{\x_{0} \in \mathcal{C}_{r}} \rho_{x_{0},t} (x).
\] 
By construction the associated probability measure satisfies the properties of $\nu$ in \cref{as:minorization:inf}.
\qed\end{proof}

We next consider the case $\qDomain =\RR^{n}$. The following \cref{prop:lya:R} shows the existence of a suitable class of Lyapunov functions.
\begin{lemma}\label{prop:lya:R}
Let $\qDomain = \mathbb{R}^{n}$. If
\begin{enumerate}[label=(\roman*)]
\item\label{item:posdef:GammaQ} $-\Gammabf \in \RR^{(n+m)\times(n+m)}$ is a stable matrix and $\Sigmabf \in \RR^{(n+m)\times(n+m)}$ such that 
\[
\Gammabf_{2,2}\Q+\Q \Gammabf_{2,2}^{\trans}
\]
is positive definite 
with $\Q$ as speficied in \cref{as:gle:q},
\item the force ${\bm F} \in \mathcal{C}^{\infty}(\mathbb{R}^{n},\mathbb{R}^{n})$ satisfies \cref{as:potential:unbounded}.
\end{enumerate}
Furthermore, if either
\begin{enumerate}[label=(\roman*),resume]
\item $\Gammabf_{1,1}$ is positive definite, 
\end{enumerate}
or
\begin{enumerate}[label=(\roman*),resume]
\item  the force ${\bm F}$ satisfies \cref{as:potential:unbounded:2}, 
\end{enumerate}
then
\begin{equation}\label{eq:lyapunov:def}
\mathcal{K}_{l}(\q,\p,\s) = \left ( \z^{\trans} \C_{A,B} \z + \norm{\q}_{2}^{2} + 2\langle \p, \q \rangle + B D( V(\q) -  u_{\min} ) + 1 \right )^{l}, ~ l \in \mathbb{N},
\end{equation}
where
\[
\C_{A,B} =
\begin{pmatrix}
B \I_{n} & A\Gammabf_{2,1}^{\trans} \\
A\Gammabf_{2,1} & B \Q^{-1}   
\end{pmatrix}\in \RR^{(n+m)\times(n+m)},
\] 
is a symmetric positive definite matrix for suitably chosen scalars $A,B >0$, and $V \in \mathcal{C}^{\infty}(\RR^{n},\RR)$ as specified in \cref{as:potential:unbounded}, defines a family of Lyapunov functions for the differential operator $\Lcgq$, i.e., for each $l \in \mathbb{N}$ there exist constants $a_{l}>0$, $b_{l}\in \RR$, 
such that for $\Lc = \Lcgq, \Kc = \Kc_{l}$, \cref{as:lyapunov:inf} holds for $a = a_{l}, b=b_{l}$.
\end{lemma}
\begin{proof}
Rewriting $\mathcal{K}_{l}$ as 
\[
\mathcal{K}_{l} (\q,\p,\s) = \left ( \x^{\trans} \hat{\C}_{A,B} \x + B D( V(\q) -  u_{\min} ) + 1 \right )^{l}, ~ l \in \mathbb{N},
\]
where 
\[
\hat{\C}_{A,B} =
\begin{pmatrix}
\I_{n} & \I_{n} & \0 \\
\I_{n} & B \I_{n} & A\Gammabf_{2,1}^{\trans} \\
\0 & A\Gammabf_{2,1} & B \Q^{-1}   
\end{pmatrix}\in \RR^{(n+m)\times(n+m)},
\]
we find by successive application of \cref{lem:schur:posdef}, that for any $A' \geq 0$ there exists $B'>0$ so that for $A=A'$ and $B \geq B'$ the matrix $\hat{\C}_{A,B}$ is positive definite and thus $\mathcal{K}_{l}\geq 1$ and $\mathcal{K}_{l}(\x) \rightarrow \infty$ as $\norm{\x} \rightarrow \infty$.
We first consider the case $l=1$. Define
\begin{equation}\label{eq:def:G:matrix}
{\bm G} := 
\begin{pmatrix}
\0 & -\I_{n} & \0 \\
 \I_{n} & \Gammabf_{1,1} & \Gammabf_{1,2} \\
\0 & \Gammabf_{2,1} & \Gammabf_{2,2}   
\end{pmatrix} \in \mathbb{R}^{(2n+m) \times (2n+m)},
\end{equation}
and
\[
\tilde{\Q} := 
\begin{pmatrix}
\I_{n} & \0 \\
\0 & \Q
\end{pmatrix},
\]
we find
\begin{equation*}
\begin{aligned}
\Lcgq\Kc_{1}(\x)  =& -
(
-\left [ {\bm F}(\q) \right ]^{\trans},  \p^{\trans},  \s^{\trans}
)
\G^{\trans} 
\hat{\C}_{A,B}
\x
+ DB \I_{n}\p \cdot \nabla_{\q} V(\q)\\
&+ \frac{\beta^{-1}}{2}\nabla_{\z} \cdot \left ( \Sigmabf \Sigmabf^{\trans} \nabla_{\z} (\z \tilde{\Q}^{-1}\z) \right ),
\end{aligned}
\end{equation*}
with
\[
\G^{\trans} 
\C 
=
-
\left(
\begin{array}{ccc}
 \I_{n} & B\I_{n} & \Gammabf_{2,1} \\
-\I_{n} + \Gammabf _{1,1} & - \I_{n} + B \Gammabf _{1,1} + A\Gammabf_{2,1}^{\trans}\Gammabf_{2,1}&  B\Q^{-1} \Gammabf _{2,1}^{\trans} \\
  \Gammabf _{1,2}^{\trans} & \Gammabf_{2,1}\Gammabf_{2,2}+ B \Gammabf _{1,2}^{\trans}& \Gammabf_{2,1}\Gammabf_{1,2} + B\Q^{-1} \Gammabf^{\trans} _{2,2} \\
\end{array}
\right)
.
\]
Hence, by virtue of \eqref{cond:potential} and \cref{as:potential:unbounded} \ref{as:potential:it:0},
\begin{equation}\label{eq:unbound:quadratic:form}
\begin{aligned}
&\Lc_{GLE}\Kc_{1}(\x) \leq\\
&\\
&
-\x^{\trans}
\underbrace{
\left(
\begin{array}{ccc}
 E\I_{n} & \0 & \0 \\
(-\I_{n} + \Gammabf _{1,1}) & - \I_{n} + B \Gammabf _{1,1} + A\Gammabf_{2,1}^{\trans}\Gammabf_{2,1} &  B\Q^{-1} \Gammabf _{2,1}^{\trans} \\
  \Gammabf _{1,2}^{\trans} & A\Gammabf_{2,1}\Gammabf_{2,2}+ B \Gammabf _{1,2}^{\trans} & A\Gammabf_{2,1}\Gammabf_{1,2} + B\Q^{-1} \Gammabf^{\trans} _{2,2} \\
\end{array}
\right)
}
_\textrm{{$=: \widehat{\bm R}_{A,B}$}}
\x\\
&\\
& - A \nabla_{\q}V(\q)^{\trans}\Gammabf_{2,1}^{\trans}\s
+ F
+ \frac{\beta^{-1}}{2}  \sum_{i,j}[ \tilde{\Q}^{-1} \Sigmabf \Sigmabf ^{\trans} \tilde{\Q}^{-1} ]_{i,j}.
\end{aligned}
\end{equation}
In order to show the existence of constants $a_{1}$ and $b_{1}$ such that the respective Lyapunov inequality satisfied, one needs to show that the right hand side of the above inequality \cref{eq:unbound:quadratic:form} can be bounded from above by a negative definite quadratic form. 
\paragraph{Case ${\rm rank}(\Gammabf_{1,1})=n$:} Let $A=0$.
In this case it is sufficient to show that the symmetric part
\[
\widehat{\bm R}^{S}_{A,B} = \frac{1}{2} \left ( \widehat{\bm R}_{A,B}+\widehat{\bm R}_{A,B}^{\trans} \right ) 
\]
 of $\widehat{\bm R}_{A,B}$ is positive definite.
 The lower right block 
\[
\left  [\widehat{\bm R}^{S}_{A,B} \right ]_{{(n+1): (2n+m)},{(n+1): (2n+m)}} = -\I_{n} + \frac{B}{2} \left (  \Gammabf  \tilde{\Q} + \tilde{\Q} \Gammabf^{\trans} \right ) \in \RR^{(n+m)\times(n+m)},
\]
of $\widehat{\bm R}^{s}_{0,B}$ is positive definite for sufficiently large $B>0$. In particular 
\[
\min \sigma \left (\left  [\widehat{\bm R}^{S}_{A,B} \right ]_{{(n+1): (2n+m)},{(n+1): (2n+m)}} \right ) = \bigO(B),
\]
as $B\rightarrow \infty$. Thus, by virtue of \cref{lem:schur:posdef} for $E>0$ there is a $B'>0$ such that $\widehat{\bm R}^{s}_{0,B}$ is indeed positive definite for all $B\geq B^{'}$.

\paragraph{Case $\Gammabf_{1,1} = \0$:} 
If \cref{as:potential:unbounded:2} holds, then by \cref{rem:on:as:potential:unbounded:2} this implies that there are values $ \overline{H}>0$ and $ \overline{h}\in \RR$ so that
\[
\abs{ \langle {\bm g}, {\bm F}(\q)\rangle } \leq   \overline{H}\abs{\langle {\bm g},\q \rangle} + \overline{h}.
\]
Therefore, it is sufficient to show that there are constants $A,B,E$ so that the function
\begin{equation}\label{eq:max:quadratic}
\begin{aligned}
\varphi(\x) &= \max \left ( - \x^{\trans} \widehat{\bm R}_{A,B}\x - A \overline{H} \q^{\trans}\Gammabf_{2,1}^{\trans}\s, ~ - \x^{\trans} \widehat{\bm R}_{A,B}\x + A\overline{H} \q^{\trans}\Gammabf_{2,1}^{\trans}\s \right )\\
&= \max_{i=1,2} -\x^{\trans}\tilde{\bm R}^{(i)}_{A,B,E} \x,
\end{aligned}
\end{equation}
can be bounded from above by a negative definite quadratic form. This means that we have to show that for suitable constants $A,B,E>0$ the symmetric part of the 
matrix 
\[
\tilde{\bm R}^{(i)}_{A,B,E} = 
\left(
\begin{array}{ccc}
 E\I_{n} & \0 & (-1)^{i} A \overline{H}\Gammabf_{2,1}^{\trans}\ \\
-\I_{n}  & - \I_{n} + A\Gammabf_{2,1}^{\trans}\Gammabf_{2,1} & \0 \\
  \Gammabf _{1,2}^{\trans} & A\Gammabf_{2,1}\Gammabf_{2,2}& A\Gammabf_{2,1}\Gammabf_{1,2} + B\Q^{-1} \Gammabf^{\trans} _{2,2} \\
\end{array}
\right),
\]
is positive definite for $i \in \{0,1\}$. (Note that we used $\Gammabf _{1,2}^{\trans}-\Q^{-1} \Gammabf _{2,1}=\0$ in the derivation of the form of  $\tilde{\bm R}^{i}_{A,B}$.)
Since $\Gammabf_{2,1}^{\trans}\Gammabf_{2,1}$ is positive definite we can choose $A$ sufficiently large so that $- \I_{n} + A\Gammabf_{2,1}^{\trans}\Gammabf_{2,1} $ is positive definite. The positive definiteness of the symmetric part of $\tilde{\bm R}^{(i)}_{A,B,E}$, $i \in \{ 0,1\}$ follows  for sufficiently large $B>0$ and $E>0$ by successive application of \cref{lem:schur:posdef}.\\

For $l>1$ we find:
\begin{equation}
\begin{aligned}
(\mathcal{L}_{H} + \mathcal{L}_{O})\mathcal{K}_{l}(\x) =& ~ l \mathcal{K}_{l-1}(\x)\mathcal{L}_{H}\mathcal{K}_{1}(\x) + l \mathcal{K}_{l-1}(\x)(-\Gammabf^{\trans} \z \cdot \nabla_{\z} \mathcal{K}_{1}(\x))\\
 &+ l \frac{\beta^{-1}}{2} \nabla_{\z } \cdot \left ( \Sigmabf \Sigmabf^{\trans} \nabla_{\z} \mathcal{K}_{1}(\x)\mathcal{K}_{l-1}(\x) \right ) \\ 
=& ~
- l \mathcal{K}_{l-1}(\x)( \z^{\trans}\Gammabf^{\trans}\tilde{\Q}\z) + l \beta^{-1}\sum_{i,j}\left[\Sigmabf \Sigmabf^{\trans}\tilde{\Q} \right]_{i,j} \mathcal{K}_{l-1}(\x)\\
&  ~ 
+ 2 l (l-1) \beta^{-1}\z^{\trans}\tilde{\Q}\Sigmabf\Sigmabf^{\trans} \tilde{\Q} \z \mathcal{K}_{l-2}(\x)\\
\leq & ~
 - l \mathcal{K}_{l-1}(\x) \left( (\Lc_{H}+\Lc_{O}) \Kc_{1}(\x)  + c_{2} \right )  \\
\leq & ~
l \mathcal{K}_{l-1}(\x) \left (  -a_{1} \mathcal{K}_{1}(\x) + b_{1} + c_{2} \right ) \\
\leq & ~
l \left ( -a_{1}  \mathcal{K}_{l}(\x) + \frac{b_{1}+c_{2}}{\epsilon_{l}^{l-1}}  + \epsilon_{l}\mathcal{K}_{l} \right )
=  -a_{l} \mathcal{K}_{l}(\x) + b_{l},
\end{aligned}
\end{equation}
with
\[
c_{2} =  - \beta^{-1}\sum_{i,j}\left[\tilde{\Q}\Sigmabf \Sigmabf^{\trans}\tilde{\Q} \right]_{i,j} +  \beta^{-1}\sum_{i,j}\left[\Sigmabf \Sigmabf^{\trans}\tilde{\Q} \right]_{i,j} 
\]
and
\[
a_{l} := l(a_{1}-\epsilon_{l}), ~
b_{l} :=   l \frac{b_{1}+c_{2}}{\epsilon_{l}^{l-1}}
\]
where $\epsilon_{l} > 0$ sufficiently small so that $a_{l}>0$.
\qed
\qed\end{proof}
We mention that \cref{as:potential:unbounded} is commonly also required for the construction of suitable Lyapunov functions in the case of the underdamped Langevin equation if $\qDomain$ is unbounded. \cref{as:potential:unbounded:2}  is an additional constraint on the potential function $U$, which is not required in the case of the underdamped Langevin equation. It is therefore not surprising that this assumption can be dropped if the noise process $\etabf$ in the GLE contains a nondegenerate white noise component.\\

If $\Sigmabf$ has full rank the minorization can be demonstrated using a simple control argument.
\begin{lemma}\label{prop:gle:minorization:3}
Let $\qDomain =\RR^{n}$. 
If ${\rm rank}(\Sigmabf) = n+m$, then \cref{eq:gle:q} satisfies a minorization condition (\cref{as:minorization:inf}). 
\end{lemma}
\begin{proof}
Note that by \cref{prop:cond:hyp}, \ref{item:cond:hyp:2} ${\rm rank}(\Sigmabf) = n+m$ immediately implies that the SDE satisfies the parabolic H\"ormander condition. Since $\Sigmabf$ is invertible, we can easily solve the associated control problem which then by \cref{lem:HiMaSt} implies that a minorization condition is satisfied. The proof  of the existence of a suitable control is essentially the same as in the case of the under-damped Langevin equation (see e.g. \cite{Mattingly2002}): Let $T>0$ and $(\q^{-},\p^{-},\s^{-}),(\q^{+},\p^{+},\s^{+}) \in \RR^{2n+m}$. We need to show that there exists $u \in L^{1}([0,T], \RR^{m})$, solving the control problem
\begin{equation}\label{eq:control:2nd:order}
\begin{aligned}
\dot{\q} & =  \p,\\
\dot{\p} & =  {\bm F}(\q)  - \Gammabf_{1,1} \p + \Gammabf_{1,2} \s +  \Sigmabf_{1} {\bm u},\\
\dot{\s} & =   - \Gammabf_{2,1} \p + \Gammabf_{2,2} \s +  \Sigmabf_{2} {\bm u},\\
\end{aligned}
\end{equation}
subject to 
\[
(\q(0),\p(0),\s(0))= (\q^{-},\p^{-},\s^{-}),~(\q(T),\p(T),\s(T))= (\q^{+},\p^{+},\s^{+}).
\]
It is easy to verify that there exists a smooth path $\tilde\q \in\mathcal{C}^{2}([0,T], \RR^{n})$ and $\tilde\s \in\mathcal{C}^{2}([0,T], \RR^{m})$  such that 
\[
(\tilde\q(0),\dot{\tilde\q}(0))= (\q^{-},\p^{-}),~(\tilde\q(T),\dot{\tilde{\q}}(T))= (\q^{+},\p^{+}),
\]
and
\[
\tilde{\s}(0) = \s^{-}, ~ \tilde{\s}(T) = \s^{+}.
\]
Rewrite \eqref{eq:control:2nd:order} as a second order differential equation in $\q$ and $\s$:
\[
\begin{aligned}
\ddot{\q} &= - \nabla_{\q}U(\q) - \Gammabf_{1,2}\dot{\q} -\Gammabf_{1,2} \s +\Sigmabf_{1} {\bf u},\\
\dot{\s} &= -  \Gammabf_{2,1}\dot{\q} - \Gammabf_{2,2}\s +\Sigmabf_{2} {\bf u},
\end{aligned}
\]
thus,
\begin{equation}
{\bf u}(t) =\Sigmabf^{-1}  \begin{pmatrix} \ddot{\tilde\q}(t) + \nabla_{\q}U(\tilde\q(t)) + \Gammabf_{1,1}\dot{\tilde\q}(t)+ \Gammabf_{1,2}\tilde{\s}(t) \\
\dot{\tilde{\s}}(t)  +  \Gammabf_{2,1}\dot{\tilde{\q}}(t) +\Gammabf_{2,2}\tilde{\s}(t)
\end{pmatrix},
\end{equation}
is a solution of \eqref{eq:control:2nd:order}.
\qed\end{proof}

The following \cref{prop:minorization:R-2} shows that the minorization condition is satisfied in the case of a GLE with unbounded configurational domain and $\Gammabf_{1,1}=\0$. 
\begin{lemma}\label{prop:minorization:R-2}
Under the same conditions as \cref{thm:ergo:unbounded:3} it follows that \cref{as:minorization:inf}  is satisfied for \cref{eq:gle:q}. 
\end{lemma}
\begin{proof}
By \cref{as:potential:unbounded:2} the force ${\bm F}$ can be decomposed as
\[
{\bm F}(\q) = {\bm F}_{1}(\q) + {\bm F}_{2}(\q),
\]
where  $\norm{{\bm F}_{1}(\q)}_{\infty}$ is uniformly bounded in $\q\in\RR$ and
\[
{\bm F}_{2}(\q) = {\bm H} \q,
\]
with  ${\bm H}\in \RR^{n\times n}$ being a positive definite matrix. Consider the dynamics 
\begin{equation}
\begin{aligned}\label{gle:tilde}
\dot{\q}^{a} &= \p^{a},\\
\dot{\p}^{a} & =  - {\bm H}\q^{a} -  \Gammabf_{1,2}\s^{a},\\
\dot{{\bm g}}^{a} &= -\Gammabf_{2,1}\p^{a} - \Gammabf_{2,2}{\s}^{a}  + \frac{\beta^{-1}}{2} \Sigmabf_{2}\dot{\W},\\
&\text{with} ~~ (\q^{a}(0), \p^{a}(0), \s^{a}(0) ) =\x_{0},
\end{aligned}
\end{equation}
where $\x_{0}\in \RR^{2n+m}$. The solution of \cref{gle:tilde} is Gaussian hence 
\[
\mu^{a}_{t}(\dd \x) = \mathcal{N}(\dd \x; {\bm \mu}_{t}, \mathscr{V}_{t}),
\]
where ${\bm \mu}_{t}\in \RR^{2n+m}$ and $ \mathscr{V}_{t} \in \RR^{(2n+m)\times(2n+m)}$. Moreover,
by \cref{prop:cond:hyp}, \ref{item:cond:hyp:3}, the SDE \cref{gle:tilde} is hypoelliptic, hence $\mathscr{V}_{t}$ is non-singular for all $t>0$. As a consequence  
\[
\text{supp}(\mu^{a}_{t}) = \xDomain
\]
for all $t>0$.
Moreover, we notice that
\[
{\bm F}_{1}(\q) = {\bm u}(\q) \Sigmabf_{2},
\]
with
\[
{\bm u}(\q)  =-{\bm F}_{1}(\q) \I_{n,m} \Sigmabf_{2}^{-1},
\]
where
\[
\I_{n,m}= \left ( \I_{n}, \0 \right ) \in \RR^{n \times m}.
\]
 Using \cref{lem:lya:exp:gle:nonconst:coeff} it follows by the same chain of arguments as in the proof of \cref{lem:novikov:cond}, that ${\bm u}$ satisfies Novikov's condition and by virtue of Girsanov's theorem the support of the law $\mu_{t}$ of the solution of  \cref{eq:gle:q} with initial condition $\x(0) = \x_{0}$ coincides with the law of $\mu^{a}_{\x_{0},t}$, i.e., $\text{supp}(\mu_{t})=\xDomain$.  Let $\mu_{\x_{0},t}(\dd \x) = \rho_{\x_{0},t}(\x)\dd \x$. As in the proof of \cref{thm:minorization:torus} we can construct a minoring measure $\eta(\dd x) = \rho( x) \dd x,$ as 
\[
\rho(x) := \min_{\x_{0} \in \mathcal{C}_{r}} \rho_{x_{0},t} (x).
\] 
where $\mathcal{C}_{r}\subset \RR^{2n+m}$ is a sufficiently large compact set.
\qed\end{proof}
\cref{lem:lya:exp:gle:nonconst:coeff} allows to conclude that Novikov's condition is satisfied under the assumptions of the preceding \cref{prop:minorization:R-2}.
\begin{lemma}\label{lem:lya:exp:gle:nonconst:coeff:2}
Let
\[
\widehat{\Kc}_{\theta}(\x) = e^{\frac{\theta}{2} \Kc_{l}(\x)}, ~ l=1, 
\]
with $\Kc_{1}$ as defined in \eqref{eq:lyapunov:def}. Under the same conditions as in \cref{prop:lya:R}, and provided that \cref{as:lyapunov:inf} holds for $\Lc=\Lcgq$, $\Kc=\Kc_{1}$, then also $\widehat{\Kc}_{\theta}$ satisfies \cref{as:lyapunov:inf} for $\Lc = \Lcgq$ and sufficiently small $\theta>0$.
\begin{proof}
A simple calculation shows
\[
\Lcgq \widehat{\Kc}_{\theta}(\x) =   \left (\theta\Lcgq  \Kc_{1} (\x) + \frac{\beta^{-1}}{2} \Big (  \theta \sum_{i,j}\Big [(\tilde{\Q} -\I_{n+m} )\widetilde{\C} \Big ]_{i,j} +  \theta^{2} \z^{\trans} \widetilde{\C}\z \Big ) \right ) \widehat{\Kc}_{\theta}(\x),
\]
with
\[
\widetilde{\C}=\tilde{\Q}^{-1} \Sigmabf \Sigmabf ^{\trans} \tilde{\Q}^{-1}.
\]
From \cref{prop:lya:R} we know $\Lcgq  \Kc_{1}(\x) = \bT\left(-\norm{\x}^{2}\right)$, thus
\[
\begin{aligned}
\Lcgq \widehat{\Kc}_{\theta}(\x) & =  \left ( - \bT \left (\theta \norm{\x}^{2} \right ) + \bT\left ((1+\theta)  \norm{\z} \right ) + \bT \left (\theta^{2} \norm{\z}^{2} \right) \right )\Kc_{\theta}(\x), \\
\end{aligned}
\]
thus for sufficiently small $\theta>0$ and suitable $b\in\RR$,
\[
\Lcgq \widehat{\Kc}_{\theta}(\x) <  - \widehat{\Kc}_{\theta}(\x) + b.
\]
\qed\end{proof}
\end{lemma}

\subsection{Technical lemmas required in the proofs of ergodicity of \cref{eq:gle:q} with non-stationary random force}\label{sec:lem:nonstat}
We first show that under the assumptions of \cref{thm:ergo:GLEq} a minorization condition is satisfied for \eqref{eq:gle:q}. For $r>0$ let in the following $C_{r}:= \{ (\q,\p,\s) : \norm{\p,\s}_{2}<r\}$.
\begin{lemma}\label{thm:erg:gle:nonconst:coeff}
Let $\qDomain=\mathbb{T}^{n}$ and $\Gammaq_{1,2},\Gammaq_{2,1}, \Gammaq_{2,2},\Sigmaq_{2} \in \mathcal{C}^{\infty}(\qDomain, {\rm GL}_{n}(\RR))$, such that $-\Gammaq(\q)$ is stable for all $\q \in \qDomain$. Let $r>0$ and $\x_{0}\in C_{r}$. For any $t>0$ the law $\mu^{\x_{0}}_{t} := e^{t\Lc^{\dagger}} \delta_{\x_{0}}$ of the solution $\x(t)$ of \eqref{eq:gle:q} with initial condition $\x(t) = \x_{0}$ has full support. In particular, \cref{as:minorization:inf} (minorization condition) holds.
\end{lemma}
\begin{proof}
Let $\x_{0} = (\q_{0}, \p_{0}, \s_{0}) \in C_{r}$ and $\tilde{\x}_{0} = (\q_{0}, \p_{0}, {\bm g}_{0})$ with ${\bm g}_{0} = \Gammaq_{1,2}(\q_{0})\s_{0}$. Consider the following cascade of modifications of \eqref{eq:gle:q}:
\begin{equation}\label{gle:c}
\begin{aligned}
\dot{\q}^{c} =&\, \p^{c},\\
\dot{\p}^{c} =&\, {\bm F}(\q) -  {\bm g}^{c}  \\
\dot{{\bm g}}^{c} =&\,  \sum_{i=1}^{n}\p^{c}_{i}\left (  \partial_{\q_{i}} \Gammaq_{1,2}(\q^{c}) \right ){\bm g}^{c} - \Gammaq_{1,2}(\q^{c})\Gammaq_{2,1}(\q^{c}) \p^{c}\\
 &- \Gammaq_{1,2}(\q^{c})\Gammaq_{2,2}(\q^{c})\Gammaq_{1,2}^{-1}(\q^{c}) {\bm g}^{c} + \Gammaq_{1,2}(\q^{c})\Sigmaq_{2}(\q^{c}) \dot{\W}_{t},\\
\text{with}& ~~ (\q^{c}(0), \p^{c}(0),{\bm g}^{c}(0) ) = \tilde{\x}_{0} ,
\end{aligned}
\end{equation}
and
\begin{equation}\label{gle:b}
\begin{aligned}
\dot{\q}^{b} &=\p^{b},\\
\dot{\p}^{b} & =  {\bm F}(\q^{b}) -  {\bm g}^{b}, \\
\dot{{\bm g}}^{b} &= \p^{b} - {\bm g}^{b}+ \Gammaq_{1,2}(\q)\Sigmaq_{2}(\q^{b}) \dot{\W}_{t},\\
&\text{with} ~~ (\q^{b}(0), \p^{b}(0),{\bm g}^{b}(0) ) = \tilde{\x}_{0} ,
\end{aligned}
\end{equation}
and 
\begin{equation}
\begin{aligned}\label{gle:a}
\dot{\q}^{a} &=\p^{a},\\
\dot{\p}^{a} & = {\bm F}(\q^{a}) -  {\bm g}^{a},\\
\dot{{\bm g}}^{a} &= \p^{a} - {\bm g}^{a}  + \dot{\W},\\
&\text{with} ~~ (\q^{a}(0), \p^{a}(0),{\bm g}^{a}(0) ) =\tilde{\x}_{0} .
\end{aligned}
\end{equation}
Let $\mu^{a}_{t},\mu^{b}_{t},\mu^{c}_{t}$ denote the law of the solution of \eqref{gle:a}, \eqref{gle:b} and \eqref{gle:c}, respectively. We show that for any $t>0$
\begin{enumerate}[label=(\roman*)]
\item \label{it:lem:min:1} $\text{supp}(\mu^{a}_{t}) = \xDomain$,
\item \label{it:lem:min:2} $\text{supp}(\mu^{b}_{t})= \text{supp}(\mu^{a}_{t})$,
\item \label{it:lem:min:3} $\text{supp}(\mu^{c}_{t})= \text{supp}(\mu^{b}_{t})$,
\item \label{it:lem:min:4} $\text{supp}(\mu_{t})= \text{supp}(\mu^{c}_{t})$,
\end{enumerate}
which then immediately implies that  $\text{supp}(\mu_{t})= \xDomain$ for $t>0$ and the minorization condition follows by the same arguments as in the proof of \cref{thm:minorization:torus}.
\begin{itemize}
\item 
Regarding \ref{it:lem:min:1}: the system \eqref{gle:a} satisfies the condition of \cref{thm:minorization:torus}, hence for sufficiently large $t'>0$ the law of \eqref{gle:a} at 
times $t\geq t'$ has full support. 
\item Regarding \ref{it:lem:min:2}: since $\Gammaq_{1,2}(\q)\Sigmaq_{2}(\q)$ is invertible, the controllability properties of 
\cref{gle:b} are identical to the controllability properties of \eqref{gle:a}, hence as a consequence of the Strook-Varadhan 
support theorem \cite{stroock1972support}  the law of \eqref{gle:b} and the law of \cref{gle:a} at time $t'$ coincide. In particular, together with \ref{it:lem:min:1} $\text{supp}(\mu^{c}_{t})= \text{supp}(\mu^{b}_{t})=\xDomain$. 
\item Regarding \ref{it:lem:min:3}:
We show this using \cref{thm:girsanov:1} (Girsanov's theorem). 
The difference of the drift terms in \cref{gle:b} and \cref{gle:c} can be written as
\[
\Gammaq_{1,2}(\q^{c})\Sigmaq_{2}(\q^{c}) {\bm u}(\q,\p,{\bm g}),
\]
 with  ${\bm u}(\q,\p,{\bm g})$ as defined in \cref{eq:u:func}. By \cref{lem:novikov:cond} the function ${\bm u}$ satisfies Novikov's condition \eqref{eq:novikov:cond}, which means that \cref{thm:girsanov:1} (Girsanov's theorem) is applicable and it follows that the support of the solution of \cref{gle:b} at $t'$ coincides with the support of the solution of \cref{gle:c} at $t'$, i.e., $\text{supp}(\mu^{c}_{t})= \text{supp}(\mu^{b}_{t})= \xDomain$.
\item Regarding \ref{it:lem:min:4}: We first note that since \ref{it:lem:min:1}-\ref{it:lem:min:3} holds, it trivially follows that $\mu^{c}_{t}=\xDomain$. Applying the change of variables $\s = \Gammaq_{1,2}^{-1}(\q){\bm g}$ to \eqref{gle:c} we obtain \eqref{eq:gle:q}, which means that $\mu_{t}$ is the push-forward of $\mu_{t}^{c}$ under the map,
\[
f: \begin{pmatrix} \q \\ \p \\ {\bm g} \end{pmatrix}  \mapsto \begin{pmatrix} \q \\ \p \\ \Gammaq_{1,2}^{-1}(\q){\bm g} \end{pmatrix} ,
\]
i.e.,
\[
\mu_{t}(A) = f (\mu^{c}_{t}) (A) = \mu^{c}_{t} \left ( f^{-1}(A)  \right ),~ A \in \mathcal{B}(\xDomain).
\]
Since $f$ is a smooth one-to-one mapping, in particular surjective, and $\text{supp}  (\mu^{c}_{t})= \xDomain$ we have
\[
\text{supp}(\mu_{t}) =  \text{supp} \left ( f (\mu^{c}_{t}) \right )= \xDomain.
\]
\end{itemize}
\qed\end{proof}

The following lemma, \cref{lem:novikov:cond},  shows that Novikov's condition is satisfied for the function $u$ required for the application of Girsanov's theorem in the above proof of \cref{thm:erg:gle:nonconst:coeff}.
\begin{lemma}\label{lem:novikov:cond}
Let $\qDomain = \mathbb{T}^{n}$ and  $\Gammaq$ and $\Sigmaq$ as in \cref{thm:erg:gle:nonconst:coeff}.
Define 
\begin{equation}\label{eq:u:func:1}
\begin{aligned}
{\bm u}_{1}(\q,\p,{\bm g}) =& \left (\Gammaq_{1,2}(\q)\Sigmaq_{2}(\q) \right )^{-1} \left (   \Gammaq_{1,2}(\q)\Gammaq_{2,1}(\q) \p -\p - \Gammaq_{1,2}(\q)\Gammaq_{2,2}(\q)\Gammaq_{1,2}^{-1}(\q) {\bm g} + {\bm g}   \right)\\
=&\; {\bm G}(\q)\begin{pmatrix}\p \\ {\bm g} \end{pmatrix},
\end{aligned}
\end{equation}
with
\[
{\bm G}(\q) := 
\left (\Gammaq_{1,2}(\q)\Sigmaq_{2}(\q) \right )^{-1}  
\begin{pmatrix}
\Gammaq_{1,2}(\q)\Gammaq_{2,1}(\q) - \I_{n} &  - \Gammaq_{1,2}(\q)\Gammaq_{2,2}(\q)\Gammaq_{1,2}^{-1}(\q) + \I_{n}
\end{pmatrix} \in \RR^{n \times 2n},
\]
and 
\begin{equation}\label{eq:u:func:2}
\begin{aligned}
{\bm u}_{2}(\q,\p,{\bm g})=-  \left (\Gammaq_{1,2}(\q)\Sigmaq_{2}(\q) \right )^{-1} \sum_{i=1}^{n}\p_{i}\left (  \partial_{\q_{i}} \Gammaq_{1,2}(\q) \right ){\bm g},
\end{aligned}
\end{equation}
The function 
\begin{equation}\label{eq:u:func}
{\bm u}(\q,\p,{\bm g}) = {\bm u}_{1}(\q,\p,{\bm g}) + {\bm u}_{2}(\q,\p,{\bm g})
\end{equation}
satisfies Novikov's condition \eqref{eq:novikov:cond}.\\
\end{lemma}
\begin{proof}[Proof of \cref{lem:novikov:cond}]
Since 
\[
\norm{{\bm u}_{1} + {\bm u}_{2}}_{2}^{2} \leq 2\norm{{\bm u}_{1}}_{2}^{2} + 2\norm{{\bm u}_{2}}_{2}^{2},
\]
it is sufficient to show that Novikov's condition holds for ${\bm u}_{1}$ and ${\bm u}_{2}$. We only show the validity of Novikov's condition explicitly for ${\bm u}_{1}$. \footnote{The respective proof for ${\bm u}_{2}$ is essentially the same with the only difference that in \cref{eq:bound:u1} we need to bound $\norm{{\bm u}_{2}}_{2}^{2}$ by a term proportional to $\norm{\p}^{4}_{2} +\norm{\bm g}_{2}^{4}$ instead of bounding $u_{2}$ by a term  which is proportional to $\norm{\p}^{2}_{2} +\norm{{\bm g}}_{2}^{2}$ as we do in the proof for ${\bm u}_{1}$. By choosing $l=2$ in  \cref{eq:Kc:alpha:2} the remaining steps of the proof are then exactly the same as for $u_{1}$.}
\\

Since $\Gammaq_{1,2},\Gammaq_{2,1}, \Gammaq_{2,2}$ and $\Sigmaq_{2}$ are smooth functions of $\q$ and since $\qDomain$ is compact, the spectrum of ${\bm G}^{\trans}(\q) {\bm G}(\q)$ is uniformly bounded from above in $\q$, hence there is $\lambda_{\max} >0$ such that 
\begin{equation}\label{eq:bound:u1}
\lambda_{\max}^{2} ( \norm{\p}_{2}^{2} + \norm{{\bm g}}^{2} ) \geq (\p^{\trans},{\bm g}^{\trans}) {\bm G}^{\trans}(\q) {\bm G}(\q) \begin{pmatrix} \p \\ {\bm g} \end{pmatrix} = \norm{{\bm u}_{1}(\q, \p,{\bm g})}^{2},
\end{equation}
and therefore
\[
\EE \left [ \exp(\int_{0}^{T} \norm{ {\bm u}_{1}(\q(t),\p(t),{\bm g}(t))} \dd t ) \right ]  \leq \EE \left [ \exp(\int_{0}^{T} \lambda_{\max}^{2} ( \norm{\p(t)}^{2} + \norm{{\bm g}(t)}^{2} ) \dd t ) \right ],
\]
for any $T>0$. Let $\epsilon<2\tilde{\theta}/ \lambda_{\max}^{2}$, with $\tilde{\theta} = \theta/\tilde{\lambda}_{\max}$ and $\theta>0,\tilde{\lambda}_{\max}$ 
as defined in \cref{lem:lya:exp:gle:nonconst:coeff}. We find
\[
\begin{aligned}
 \exp(\int_{0}^{T} \lambda_{\max}^{2} ( \norm{\p(t)}^{2} + \norm{{\bm g}(t)}^{2} ) \dd t )  
 =  \exp( \frac{1}{\epsilon} \int_{0}^{T} \epsilon \lambda_{\max}^{2} ( \norm{\p(t)}^{2} + \norm{{\bm g}(t)}^{2} ) \dd t )  \\
 \leq \frac{1}{\epsilon} \int_{0}^{T}  \exp( \epsilon \lambda_{\max}^{2} ( \norm{\p(t)}^{2} + \norm{{\bm g}(t)}^{2} ) ) \dd t, 
 \end{aligned}
\]
by Jensen's inequality, thus 
\[
\begin{aligned}
\EE \left [ \exp(\int_{0}^{T} \norm{ {\bm u}_{1}(\q(t),\p(t),{\bm g}(t))} \dd t ) \right ] 
&\leq \EE \left [ \frac{1}{\epsilon} \int_{0}^{T}  \exp( \epsilon \lambda_{\max}^{2} ( \norm{\p(t)}^{2} + \norm{{\bm g}(t)}^{2} ) ) \dd t   \right ] \\
&=   \frac{1}{\epsilon} \int_{0}^{T} \EE \left [ \exp( \epsilon \lambda_{\max}^{2} ( \norm{\p(t)}^{2} + \norm{{\bm g}(t)}^{2} ) ) \right ]\dd t,
\end{aligned}
\]
by Tonelli's theorem. Let for $\alpha>0$, 
\begin{equation}\label{eq:Kc:alpha:2}
\Kc_{\alpha} :=\Kc_{\alpha,l}, \,l=1,
\end{equation}
with  $\Kc_{\alpha,l}$ as defined in
\cref{eq:Kc:alpha}. Using 
\begin{equation}
\exp( \epsilon \lambda_{\max}^{2} ( \norm{\p}^{2} + \norm{{\bm g}}^{2} ) ) \leq \Kc_{\tilde{\theta}}(\z),
\end{equation}
we conclude using  \cref{lem:lya:exp:gle:nonconst:coeff}, \eqref{eq:lem:bound}
\[
\begin{aligned}
\frac{1}{\epsilon} \int_{0}^{T} \EE \left [ \exp(\epsilon \lambda_{\max}^{2} ( \norm{\p(t)}^{2} + \norm{{\bm g}(t)}^{2} ) ) \right ]  \dd t  
&\leq \frac{1}{\epsilon} \int_{0}^{T} \EE \left [ 
\Kc_{\tilde{\theta}}(\z(t)) \right ]  \dd t \\
&\leq  \frac{1}{\epsilon} \int_{0}^{T} 
e^{-t}\Kc_{\theta}(\p_{0}, \Gammaq_{1,2}(\q_{0}){\bm g}_{0}) + b (1-e^{-t}) \dd t \\
& < \infty.
\end{aligned}
\]
with $b>0$ as specified in \cref{lem:lya:exp:gle:nonconst:coeff}.
\qed\end{proof}





\begin{lemma}\label{lem:lya:exp:gle:nonconst:coeff}
Let $\qDomain = \mathbb{T}^{n}$ and  $\Gammaq$ and $\Sigmaq$ as in \cref{thm:erg:gle:nonconst:coeff} and let $\C\in \RR^{2n\times 2n}$ with 
\[
\min\sigma(\C)=1,
\]
be a symmetric positive definite matrix such that 
\[
\Gammaq^{\trans}(\q)\C + \C  \Gammaq(\q),
\]
is positive definite for all $\q \in\qDomain$. For $\alpha>0$ and $l\in\mathbb{N}$ define 
\begin{equation}\label{eq:Kc:alpha}
\Kc_{\alpha,l}(\p,\s) = e^{\frac{\alpha}{2} \left(\z^{\trans} \C \z \right)^{l}}.
\end{equation}
There exists $\theta>0$ such that
\cref{as:lyapunov:inf} is satisfied with $\Kc=\Kc_{\theta,l}$ and $\Lc=\Lcgq$. Moreover, 
for $\tilde{\theta} = \theta/ \tilde{\lambda}_{\max}$ with
\[
\tilde{\lambda}_{\max} := \max_{\q \in \qDomain} \left  \{ \abs{\lambda} \given \lambda \in \sigma \left(\Gammaq_{1,2}^{-1}(\q) \right )\right \} 
\]
the expectation of $\Kc_{\tilde{\theta},l}$
as function of the solution $(\q^{c}, \p^{c}, {\bm g}^{c})$ of \cref{gle:c} can be bounded as
\begin{equation}\label{eq:lem:bound}
\EE \left [   \Kc_{\tilde{\theta},l}(\p^{c},{\bm g}^{c}) \given (\p^{c}(0),{\bm g}^{c}(0) ) = (\p_{0},{\bm g}_{0}) \right ] 
\leq e^{-t} \Kc_{\theta,l}(\p_{0},\Gammaq_{1,2}(\q_{0}){\bm g}_{0}) + b(1- e^{-t}) + c(l,t),
\end{equation}
where $b>0$ as above and $c(l,t)$ is a finite nonnegative constant which depends on $l$ and $t$ with $c(l,t)=0$ for $l=1$ and all $t\geq0$.
\end{lemma}
\begin{proof}
We recall that the generator of \eqref{eq:gle:q} is of the form
\[
\begin{aligned}
\Lcgq = {\bm F}(\q) \cdot \nabla_{\p} + \p \cdot \nabla_{\q} - \Gammaq(\q) \z \cdot \nabla_{\z} + \frac{1}{2}\Sigmaq(\q)\Sigmaq^{\trans}(\q) : \nabla_{\p}^{2},
\end{aligned}
\]
We show the result only for the case $l=1$. For $l>1$ the result follows by induction. Let $\Kc_{\theta}=\Kc_{\theta,1}$. Applying the generator on $\Kc_{\theta}$ we obtain
\[
\begin{aligned}
\Lc \Kc_{\theta}(\p,\s)  &=  \left ( \theta {\bm F}(\q) \cdot  \left ( \C_{1,1}\p + \C_{1,2}\s \right )  \right ) \Kc_{\theta}(\p,\s) \\
& + \left ( -  \theta \Gammaq(\q) \z \cdot \C \z  + \frac{1}{2} \left (  \theta \text{tr}\left (\Sigmaq(\q)\Sigmaq^{\trans}(\q)\C \right )  + \theta^{2} \z^{\trans} \C \Sigmaq(\q)\Sigmaq^{\trans}(\q) \C\z \right ) \right )\Kc_{\theta}(\p,\s) \\
& =  \left ( - \bT \left (\theta \norm{z}^{2} \right ) + \bT \left ((1+\theta)  \norm{z} \right ) + \bT \left (\theta^{2} \norm{z}^{2} \right) \right )\Kc_{\theta}(\p,\s) \\
& <  - \Kc_{\theta}(\p,\s) + b,
\end{aligned} 
\]
for sufficiently small $\theta > 0$ and sufficiently large $b>0$. Consequently, for $\tilde{\theta} = \theta/ \tilde{\lambda}_{\max}$, we obtain
\[
\begin{aligned}
&\,\EE \left [   \Kc_{\tilde{\theta}}(\p^{c}(t),{\bm g}^{c}(t)) \given (\p^{c}(0),{\bm g}^{c}(0) ) = (\p_{0},{\bm g}_{0}) \right ] \\
=&\,\EE \left [   \Kc_{\tilde{\theta}}(\p(t),\Gammaq_{1,2}^{-1}(\q(t)) \s(t)) \given (\p(0),\s(0) ) = (\p_{0},\Gammaq_{1,2}(\q_{0}){\bm g}_{0}) \right ] \\
\leq& \,\EE \left [   \Kc_{\tilde{\theta}}(\tilde{\lambda}_{\max}\p(t), \tilde{\lambda}_{\max} \s(t)) \given (\p(0),\s(0) ) = (\p_{0},\Gammaq_{1,2}(\q_{0}){\bm g}_{0}) \right ] \\
=& \,\EE \left [   \Kc_{\theta}(\p(t), \s(t)) \given (\p(0),\s(0) ) = (\p_{0},\Gammaq_{1,2}(\q_{0}){\bm g}_{0}) \right ] \\
\leq&\, e^{-t} \Kc_{\theta}(\p_{0},\Gammaq_{1,2}(\q_{0}){\bm g}_{0}) + b(1- e^{-t}).
\end{aligned}
\]
\qed\end{proof}
The last \cref{prop:lya:torus-2} of this section provides conditions for the existence of suitable Lyapunov functions with polynomial growth for \eqref{eq:gle:q}.
\begin{lemma}\label{prop:lya:torus-2}
Let $\qDomain = \mathbb{T}^{n}$, $-\Gammabf \in \RR^{(m+n)\times(n+m)}$ stable, and $U \in \mathcal{C}^{\infty}(\mathbb{T}^{n},\mathbb{R})$. Moreover, assume that \eqref{as:C:matrix} holds and let $\C$ be as specified therein. 
\[
\mathcal{K}_{l}(\q,\p,\s) = \left ( \z^{\trans} \C \z +  U(\q) - U_{min}  + 1 \right )^{l}, ~ l \in \mathbb{N},
\]
defines a family of Lyapunov functions for the differential operator $\Lcgq$, i.e.,  for each $l \in \mathbb{N}$ there exist constants $a_{l}>0$, $b_{l}\in \RR$, 
such that for $\Lc = \Lcgq, \Kc = \Kc_{l}$, \cref{as:lyapunov:inf} holds for $a = a_{l}, b=b_{l}$.
\end{lemma}
\begin{proof}
The proof is very similar to the proof \cref{prop:lya:torus}. The existence of a suitable matrix $\C$ as specified in \eqref{as:C:matrix} allows to extend all arguments in that proof with only some very small adaptations. For this reason we skip a details of the proof here. 
\qed\end{proof}


\section{Conclusion}
In this article we have presented an integrated perspective on ergodic properties of the generalized Langevin equation, for systems that can be written in the quasi-Markovian form.  Although the GLE was well studied in the case of constant friction and damping and for conservative forces, our results indicate that these can often be extended to  nonequilibrium models with non-gradient forces and non-constant friction and noise, thus providing a foundation for using GLEs in a much broader range of applications.

\section*{Acknowledgements}
The authors wish to thank Greg Pavliotis (Imperial), Jonathan Mattingly (Duke) and Gabriel Stoltz (ENPC) for their generous assistance in providing comments at various stages of this project. In particular, the authors thank Jonathan Mattingly for pointing out the possibility of using Girsanov's theorem in the proof of \cref{thm:erg:gle:nonconst:coeff}.  Both authors acknowledge the support of the European Research Council (Rule Project, grant no. 320823).  BJL further acknowledges the support of the EPSRC (grant no. EP/P006175/1) during the preparation of this article. 
The work of MS was supported by the National Science Foundation under grant DMS-1638521 to the Statistical and Applied Mathematical Sciences Institute. 
%
%
\bibliographystyle{abbrv}
\bibliography{./GLErefs.bib}
%
\appendix
\section{Auxiliary material on linear algebra}
The following \cref{lem:schur:posdef} is repeatedly used in the proofs of \cref{lem:purecolor} and  \cref{prop:lya:R}, as well as in \cref{exa:not:empty} to show the positive (semi-)definiteness of symmetric matrices. 
\begin{lemma}\label{lem:schur:posdef}
Let $A$ be a symmetric block structured matrix of the form
\[
{\bm A}:=\begin{pmatrix}
{\bm A}_{1,1} & {\bm A}_{1,2} \\
{\bm A}_{1,2}^{\trans} & {\bm A}_{2,2} \\
\end{pmatrix} \in \mathbb{R}^{n+m \times n+m}
\]
\begin{enumerate}[label=(\roman*)]
\item \label{item:lem:eq:1}If ${\bm A}_{2,2}$ is positive definite, then ${\bm A}$ is positive (semi-)definite if and only if 
\[
{\bm A}_{1,1} - {\bm A}_{1,2}{\bm A}_{2,2}^{-1} {\bm A}_{1,2}^{\trans}
\]
is positive (semi-)definite
\item \label{item:lem:eq:2}If ${\bm A}_{1,1}$ is positive definite, then ${\bm A}$ is positive (semi-)definite if and only if 
\[
{\bm A}_{2,2} - {\bm A}_{1,2}^{\trans}{\bm A}_{1,1}^{-1} {\bm A}_{1,2}
\]
is positive (semi-)definite
\item \label{item:lem:eq:3} Let ${\bf A}_{2,2}^{g}$ denote a generalised inverse of ${\bf A}_{2,2}$, i.e., ${\bf A}_{2,2}^{g}$ is a $m\times m $ matrix which satisfies 
\[
{\bf A}_{2,2}{\bf A}_{2,2}^{g}{\bf A}_{2,2} = {\bf A}_{2,2}.
\]
The matrix ${\bm A}$ is positive semi-definite if and only if the matrices
$
{\bf A}_{2,2}
$
and\\
$
{{\bm A}_{1,1} - {\bm A}_{1,2}{\bm A}_{2,2}^{g} {\bm A}_{1,2}^{\trans}}
$
are positive semi-definite,
and
\[
({\bf I} - {\bf A}_{2,2} {\bf A}_{2,2}^{g}){\bf A}_{1,2}^{\trans} = \0,
\]
i.e.,
the span of the column vectors of ${\bm A}_{1,2}$ is contained in the span of the column vectors of ${\bm A}_{1,1}$.
\end{enumerate}
\end{lemma}
\begin{proof}
The statements \ref{item:lem:eq:1} and \ref{item:lem:eq:2} follow from Theorem 1.12 in \cite{Zhang2005}. Statement \ref{item:lem:eq:3} corresponds to Theorem 1.20 in the same reference. \qed
\end{proof}


\section{Auxiliary material on stochastic analysis}\label{sec:stochastic:analysis}
In this section we provide a brief overview of the general framework used in the ergodicity proofs and derivation of convergence rate in \cref{sec:ergodic:GLE}. For a comprehensive overview we refer to the review articles  \cite{Mattingly2002,Rey-Bellet2006a,Lelievre2016a}. \\

Consider an SDE defined on the domain $\xDomain = \mathbb{T}^{n_{1}} \times \mathbb{R}^{n_{2}}, n = n_{1}+n_{2}\in \mathbb{N}$ which is of the form 
\begin{equation}\label{eq:generic:SDE}
\dd X = \drift(X) \dd t+ \diffusion(X) \dd \W, ~ X(0) \sim \mu_{0},
\end{equation}
with smooth coefficients $\drift \in \mathcal{C}^{\infty}( \xDomain, \RR^{n}), \diffusion = [\diffusion_{i}]_{1\leq i \leq n} \in \mathcal{C}^{\infty}(\xDomain,\RR^{n\times n})$, and  initial distribution $\mu_{0}$. In order to simplify the presentation we further assume that the diffusion coefficient $\diffusion$ is such that the It\^o and Stratonovich interpretation of \cref{eq:generic:SDE} coincide, i.e., 
\[
\nabla \cdot \left ( \diffusion\, \diffusion^{\trans} \right )- \diffusion\, \nabla \cdot \diffusion^{\trans} \equiv \0.
\]
Let  further $\Lc$ denote the associated infinitesimal generator of \cref{eq:generic:SDE}, i.e.,
\begin{equation}
\Lc = \drift(X) \cdot \nabla  + \diffusion(X) : \nabla^{2},
\end{equation}
when considered as an operator on the core $\mathcal{C}^{\infty}(\xDomain,\RR)$, and let $\Lc^{\dagger}$ denote the formal adjoint of $\Lc$, i.e., the Fokker-Planck operator associated with the SDE \cref{eq:generic:SDE}. Furthermore, let $e^{t\Lc},e^{t\Lc^{\dagger}}$ denote the associated semigroup operators of $\Lc$, and $\Lc^{\dagger}$, respectively, i.e.,
\begin{equation}
\forall \varphi  \in \mathcal{C}^{\infty}(\xDomain,\RR) : e^{t\Lc} \varphi(x) = \EE[ \varphi(X(t)) \given X(0) = x],\footnote{The expectation is taken with respect to the Brownian motion $\W$.}
\end{equation}
for (Lebesgue-)almost all $x \in \RR^{n}$, and
\[
\int \left ( e^{t\Lc} \varphi \right )(x) \mu_{0}(\dd x) = \int \varphi(x) \left ( e^{t\Lc^{\dagger}} \mu_{0}\right ) (\dd x).
\]
\begin{definition}\label{def:K:infty}
For a given function $\Kc \in \mathcal{C}^{\infty}(\xDomain, [1,\infty))$ which is such that $\Kc(\x) \rightarrow \infty$ as $\norm{\x}\rightarrow \infty$, 
define
\begin{equation}
\norm{\varphi}_{L^{\infty}_{\Kc}} := \norm*{ \frac{\varphi}{\Kc}}_{\infty}, ~\varphi : \xDomain \rightarrow \RR~\text{measureable}.
\end{equation}
We denote by
\begin{equation}
L^{\infty}_{\Kc}(\xDomain) := \left \{  \varphi \text{ measurable } :\norm{\varphi}_{L^{\infty}_{\Kc}}<\infty  \right \}
\end{equation}
the set of measurable functions for which the ratio $\frac{\varphi}{\Kc}$ is bounded. 
\end{definition}
It can be easily verified that $\norm{\varphi}_{L^{\infty}_{\Kc}}$ defines a norm and that $ L^{\infty}_{\Kc}(\xDomain) $ equipped with the norm  $\norm{\varphi}_{L^{\infty}_{\Kc}} $ can be associated with a Banach space, which we denote by $\left ( L^{\infty}_{\Kc}(\xDomain) , \norm{\cdot}_{L^{\infty}_{\Kc}} \right )$.

Throughout this article we use Lyapunov function techniques to show (geometric) ergodicity of SDEs of the generic form \cref{eq:generic:SDE}. More specifically, we follow the standard recipe for proofs of exponential convergences of the semigroup operator $e^{t\Lc}$ in weighted $L^{\infty}$ spaces as outlined, e.g., in \cite{meyn1993stability,Mattingly2002,Rey-Bellet2006a,Lelievre2016a}, that is we show that a suitable Lyapunov condition (\cref{as:lyapunov:inf}) and a minorization condition (\cref{as:minorization:inf}) are satisfied:
\begin{assumption}[Infinitesimal Lyapunov condition]\label{as:lyapunov:inf}
There is a function $\Kc \in \mathcal{C}^{\infty}(\xDomain, [1,\infty))$ with $\lim_{\left \|\x \right\|\rightarrow \infty}\mathcal{K}(x) = \infty$, 
and real numbers $a\in (0,\infty), b \in \RR$ such that,
\begin{equation}\label{eq:lyapunov:inf}
\mathcal{L} \mathcal{K} \leq - a \mathcal{K}+b.
\end{equation}
\end{assumption}
\begin{assumption}[Minorization condition]\label{as:minorization:inf}
For some $t^{\prime}>0$ there exists a constant $\eta \in(0,1)$ and a probability measure $\nu$ such that
\[
\inf_{x\in \mathcal{C}} e^{t^{\prime}\Lc^{\dagger}}\delta_{x}(\dd  y) \geq \eta \nu(\dd y) 
\]
where $\mathcal{C} = \{ x \in \xDomain  : \; \Kc(x) \leq \Kc_{\rm max} \}$ for some 
$\Kc_{\rm max} > 1+ 2b/a,$
where $a,b$ are the same constants as in~\cref{eq:lyapunov:inf}.
\end{assumption}
If the above assumptions are satisfied, then the following proposition, which follows from the arguments in \cite{Lelievre2016a} (see also the other above mentioned references), allows to derive exponential decay estimates in the respective weighted $L^{\infty}$ space associated with the Lyapunov function $\Kc$. 
\begin{proposition}[Geometric ergodicity, \cite{Lelievre2016a}]\label{thm:geometric:ergodic}
Let \cref{as:lyapunov:inf} and \cref{as:minorization:inf} hold. The solution of the SDE \cref{eq:generic:SDE} admits a unique invariant probability measure $\pi$ such that 
\begin{enumerate}[label=(\roman*)]
\item\label{it:prop:geo:1} there exist positive constant $\lambda, \widetilde{C}$ so that for any $\varphi \in L^{\infty}_{\Kc}(\xDomain)$
\begin{equation}\label{eq:exp:conv}
\norm*{ e^{t\Lc} \varphi  -\EE_{\pi}\varphi }_{L^{\infty}_{\Kc}} \leq \widetilde{C}e^{-t\lambda} \norm*{\varphi - \EE_{\pi}\varphi }_{L^{\infty}_{\Kc}}.
\end{equation}
\item \label{it:prop:geo:2} 
\begin{equation}\label{eq:K:L1}
\int_{\xDomain} \Kc \dd \pi < \infty .
\end{equation}
\end{enumerate}
\end{proposition}
If for the solution of \cref{eq:generic:SDE}  the implications of  \cref{thm:geometric:ergodic} hold we also say 
that the solution $X$ of  \cref{eq:generic:SDE}  is {\em geometrically ergodic}. In the main body of this article we use \cref{thm:geometric:ergodic} to derive exponential decay estimates of the form \cref{eq:geo:conv:gle} in \cref{thm:ergo:bounded:1,thm:ergo:unbounded:2,thm:ergo:unbounded:3,thm:ergo:GLEq}. In these theorems \cref{as:lyapunov:inf} can be directly shown to hold by explicitly constructing a suitable Lyapunov function $\Kc$ satisfying \cref{eq:lyapunov:inf} (see \cref{prop:lya:torus,prop:lya:R,prop:lya:torus-2}). A very common way to show \cref{as:minorization:inf} is by showing (i) that the transition kernel associated with the SDE \cref{eq:generic:SDE} is smooth as specified in  \cref{as:smooth:trans}, and (ii) that the SDE \cref{eq:generic:SDE} is controllable as specified in \cref{as:control}. By virtue \cref{lem:HiMaSt} it then follows that a minorization condition holds.


\begin{assumption}\label{as:smooth:trans}
For any $t>0$ the transition kernel associated with the SDE \cref{eq:generic:SDE} possesses a density $p_{t} (x, y)$, i.e., 
\[
 \forall \,x \in \xDomain:~(e^{t\Lc^{\dagger}}\delta_{x})(A)=  \int_{A} p_{t}(x,y)dy,  ~ A \subset \xDomain, \; A \text{ measurable}.
\]
and $p_{t}(x,y)$ is jointly continuous in $(x,y)\in \xDomain \times \xDomain$.
\end{assumption}
\begin{assumption}\label{as:control}
There is a $t_{\max}>0$ so that for any $x^{-}, x^{+} \in \xDomain$, there is a $t>0$, with $t \leq t_{\max}$,so that the control problem  
\begin{equation}
\begin{aligned}
\dot{\tilde{X}} &= \drift(\tilde{X}) + \diffusion(\tilde{X}) u,
\end{aligned}
\end{equation}
subject to
\[
\tilde{X}(0)= x^{-}, \text{ and } \tilde{X}(t)=x^{+},
\]
has a smooth solution $u \in \mathcal{C}^{1}([0,t_{\max}],\xDomain)$. 
\end{assumption}
\begin{lemma}[\cite{Mattingly2002}]\label{lem:HiMaSt}
If \cref{as:smooth:trans} and \cref{as:control} are satisfied, then also \cref{as:minorization:inf} holds.
\end{lemma}
\cref{as:smooth:trans} follows directly from hypoellipticity of the operator $\partial_{t} - \Lc^{\dagger}$  (see e.g. \cite{Rey-Bellet2006a,Pavliotis2014}, for a precise definition of hypoellipticity).
A common way to establish hypoellipticity of a differential operators is via H\"ormander's theorem (\cite{hormander1985analysis}, Theorem 22.2.1, on page 353). The following proposition is an adaption of H\"ormander's theorem to the parabolic differential operator $\partial_{t} -\Lc^{\dagger}$:
\begin{proposition}\label{thm:hormander}
Let $\drift$ and $\diffusion$ be the drift coefficient and the diffusion coefficient of the SDE \cref{eq:generic:SDE}, respectively. Let $\diffusion_{0} := \drift$. Iteratively define a collection of vector fields by
\begin{equation}\label{eq:hormander:iterate}
\mathscr{V}_{0} = \{ \diffusion_{i} :  i \geq  1 \},
\hspace{0.1in} \mathscr{V}_{k+1} =\mathscr{V}_{k}  \cup \{ [{\bm v},\diffusion_{i}] : {\bm v} \in \mathscr{V}_{k}, 0 \leq i \leq n\}.
\end{equation}
where 
\[
[{\bm X}, {\bm Y}] = (\nabla {\bm Y} ) {\bm X} -( \nabla {\bm X}) {\bm Y},  
\]
denotes the commutator of vector fields ${\bm X},{\bm Y} \in \mathcal{C}^{\infty}(\xDomain,\RR^{n})$ and $(\nabla {\bm X}),(\nabla {\bm Y})$ their Jacobian matrices.
If
\begin{equation}\label{eq:hormander:cond}
\forall \x\in\mathbb{R}^{n},~~ \lspan \left \{{\bm v}(\x) : {\bm v} \in \bigcup_{k \in \mathbb{N}}\mathscr{V}_{k} \right \} = \mathbb{R}^{n},
\end{equation}
we say that the SDE \cref{eq:generic:SDE} satisfies the parabolic H\"ormander condition, and it follows that the operator $\partial_{t} - \Lc^{\dagger}$ is hypoelliptic.
\end{proposition}
We use \cref{lem:HiMaSt} in the proof of \cref{prop:gle:minorization:3} in \cref{thm:ergo:unbounded:2}. For some instances of \cref{eq:gle:q} it is not easy to construct a suitable control $u$ such that \cref{as:control} is satisfied. In these cases we either show a minorization condition by explicitly constructing the minorizing measure $\nu$ in \cref{as:minorization:inf} if the right hand side of  \cref{eq:gle:q} can be decomposed into a linear and a bounded part (see \cref{thm:ergo:bounded:1}),  or by inferring the existence of a suitable minorizing measure by showing that the support of the SDE under consideration is equivalent to the support of another SDE satisfying a minorization condition via Girsanov's theorem (\cref{prop:minorization:R-2,thm:erg:gle:nonconst:coeff}). Girsanov's theorem provides conditions under which the path measures of two It\^o processes are mutually absolutely continuous, which in particular implies that at any time $t \geq 0$  the laws of these It\^o processes are equivalent. We will use Girsanov's theorem in \cref{sec:ergodic:GLE} in order to prove the minorization condition for GLEs which in a Markovian representation possess coefficients which depend on the configurational variable. Here we provide a version of Girsanov's theorem which is adapted to It\^o-diffusion processes.
\begin{proposition}[Girsanov's theorem, \cite{Oksendal2003}] \label{thm:girsanov:1}
Consider the two It\^o diffusion processes
\begin{align}
\dd X(t) &= \drift_{x}(X) \dd t + \diffusion(X)\dd \W(t); ~ X(0) = x_{0}, \label{eq:girs:ito:1}\\
\dd Y(t) &= \drift_{y}(Y) \dd t + \diffusion(Y)\dd \W(t); ~ Y(0) = x_{0}, \label{eq:girs:ito:2}
\end{align}
where $ x_{0} \in \xDomain$, $\W$ is a standard Wiener process in $\RR^{n}$, and $\drift_{x},\drift_{y} : \xDomain \rightarrow \RR^{n}$ and $\diffusion : \xDomain \rightarrow \RR^{n \times m}, m \in \mathbb{N}$, are such that there exist unique strong solutions $X, Y$ for \eqref{eq:girs:ito:1} and \eqref{eq:girs:ito:2}, respectively. If there is a function ${\bm u} \in \mathcal{C}(\xDomain,\RR^{n})$ such that
\[
 \drift_{x} -\drift_{y} = \diffusion {\bm u} 
\]
and ${\bm u} $ satisfies {\em Novikov's} condition
\begin{equation}\label{eq:novikov:cond}
\EE \left [ \exp \left (\frac{1}{2} \int_{0}^{T} \norm{{\bm u} (X(t))}_{2}^{2} \dd s \right ) \right ] < \infty.
\end{equation}
then the path measures of $X$ and $Y$ on any finite time interval are equivalent. In particular, the support of the law of $X(t)$ and the support of the law of $Y(t)$ coincide for any $t>0$.
\end{proposition}





\end{document}